\documentclass[letterpaper]{article}

\usepackage{amsmath}
\usepackage{amsfonts}
\usepackage{amsthm}
\usepackage{amssymb}
\usepackage{datetime}
\usepackage{color}

\theoremstyle{plain}
\newtheorem{theorem}{Theorem}[section]
\newtheorem{lemma}[theorem]{Lemma}
\newtheorem{corollary}[theorem]{Corollary}
\theoremstyle{definition}
\newtheorem{definition}[theorem]{Definition}

\newtheorem{example}[theorem]{Example}
\newtheorem{fact}[theorem]{Fact}
\newtheorem{remark}[theorem]{Remark}

\numberwithin{equation}{section}

\newcommand{\A}{{\mathbb A}}
\newcommand{\all}{\hbox{for all}}

\newcommand{\bra}[2]{\langle#1,#2\rangle}

\newcommand{\Bra}[2]{\big\langle#1,#2\big\rangle}
\newcommand{\cmqd}{{CMQd}}
\newcommand{\CML}{{\cal CML}}

\newcommand{\D}{{\cal D}}

\newcommand{\dbs}{^{**}}

\newcommand{\dom}{\hbox{\rm dom}}
\newcommand{\eighth}{\ts\frac{1}{8}}

\newcommand{\eps}{\varepsilon}

\newcommand{\fourth}{\ts\frac{1}{4}}

\newcommand{\half}{{\textstyle\frac{1}{2}}}
\newcommand{\hexth}{{\textstyle\frac{1}{16}}}

\newcommand{\I}{\mathbb I}

\newcommand{\infn}{\inf\nolimits}

\newcommand{\J}{{\mathbb J}}

\newcommand{\limn}{\lim\nolimits}

\newcommand{\limsupn}{\limsup\nolimits}

\newcommand{\lr}{\Longrightarrow}
\newcommand{\Lt}{{\wt L}}

\newcommand{\MML}{{\cal MML}}

\newcommand{\minn}{\min\nolimits}
\newcommand{\NN}{\mathbb N}
\newcommand{\nth}{{\textstyle\frac{1}{n}}}
\newcommand{\on}{\hbox{on}}

\newcommand{\PC}{{\cal PC}}
\newcommand{\PCLSC}{{\cal PCLSC}}

\newcommand{\qlr}{\quad\Longrightarrow\quad}
\newcommand{\qt}{\wt q}

\newcommand{\quand}{\quad\hbox{and}\quad}
\newcommand{\rbar}{\,]{-}\infty,\infty]}

\newcommand{\RR}{\mathbb R}
\newcommand{\rl}{\Longleftarrow}
\newcommand{\rt}{\wt r}

\newcommand{\squaree}{\ \square_E\ }
\newcommand{\ssquare}{\ \square\ }
\newcommand{\st}{\hbox{such that}}
\newcommand{\tsum}{\ts\sum}
\newcommand{\sumn}{\sum\nolimits}
\newcommand{\tsumn}{\ts\sumn}
\newcommand{\supn}{\sup\nolimits}

\newcommand{\TNW}{\cal{T_{NW}}}
\newcommand{\toto}{\rightrightarrows}

\newcommand{\ts}{\textstyle}

\newcommand{\wh}{\widehat}
\newcommand{\wt}{\widetilde}
\newcommand{\W}{{\cal W}}

\newcommand{\xbra}[2]{\lfloor#1,#2\rfloor}
\newcommand{\Xbra}[2]{\big\lfloor#1,#2\big\rfloor}

\newcommand{\Cor}{Corollary~\ref}
\newcommand{\Def}{Definition~\ref}
\newcommand{\Defs}{Definitions~\ref}
\newcommand{\Ex}{Example~\ref}
\newcommand{\Lem}{Lemma~\ref}
\newcommand{\Lems}{Lemmas~\ref}

\newcommand{\Rem}{Remark~\ref}
\newcommand{\Sec}{Section~\ref}

\newcommand{\Thm}{Theorem~\ref}
\newcommand{\Thms}{Theorems~\ref}

\title{Faces of quasidensity}
\author{
Stephen Simons
\thanks{
Department of Mathematics, University of California, Santa Barbara, CA\ 93106-3080, U.S.A.
Email: \texttt{stesim38@gmail.com}.}}

\date{}

\begin{document}

\maketitle

\begin{abstract}\noindent
This paper is about the maximally monotone and quasidense subsets of the product of a real Banach space and its dual.  We discuss six subclasses of the maximal monotone sets that are equivalent to the quasidense ones.   We define the Gossez extension  to the dual of a maximally monotone set, and give nine equivalent characterizations of an element of this set in the quasidense case.   We discuss maximally monotone sets of ``type (NI)'' (one of the six classes referred to above) and we show that the ``tail operator'' is {\em not} of type (NI), but it is the Gossez extension of a maximally monotone set that {\em is} of type (NI). We  generalize Rockafellar's surjectivity theorem for maximally monotone subsets of reflexive Banach spaces to maximally monotone subsets of type (NI) of general Banach spaces.     We discuss a generalization of the Brezis-Browder theorem on monotone linear subspaces of reflexive spaces to the nonreflexive situation.    
 We also discuss briefly maximally monotone subsets of ``type (D)'' and ``type (WD)'' (two more of the six classes referred to above).
\end{abstract}




\section{Introduction}
\setlength{\belowdisplayskip}{4pt}
\setlength{\belowdisplayshortskip}{4pt}
\setlength{\abovedisplayskip}{4pt}
\setlength{\abovedisplayshortskip}{4pt}

Let $E$ be a nontrivial real Banach space with dual $E^*$, and $B := E \times E^*$.   After some notational preliminaries, the {\em quasidensity} of a subset of $B$ is defined in \Def{QDdef}.   This definition appears in condensed form in \eqref{QD1}, and in expanded form in \eqref{QD2}.   A glance at the complexity of these two forms should explain why we use the condensed form for analysis whenever possible.

\Sec{SNsec} contains some fairly standard notation. \Sec{RLsec} contains the definitions and basic properties of the three quantities $L$, $q$ and $r$ that are needed for our analysis using the ``condensed notation''.

We prove in  \Thm{RLMAXthm} the fundamental result that a {\em closed, monotone quasidense set is maximally monotone}.   This has a refreshingly simple proof.   In \Def{QDdef}, we also introduce the acronym {\em``\cmqd''}, which will be used all through this paper.   In \Rem{TAILrem}, we introduce the {\em tail operator}, which will be discussed further in \Ex{TAILNIex} and \Ex{TAILGex}.

We suppose for the rest of  this paper that $M$ is a maximally monotone subset of $B$.   Our analysis depends on three functions: $P_M\colon\ B \to [0,\infty]$ (see \eqref{P}), $F_M\colon B^* \to\ ]-\infty,\infty]$ which is defined in terms of $P_M$ \big(see \eqref{DEL}\big) and $G_M\colon B^* \to\ ]-\infty,\infty]$ \big(see \eqref{G}\big).   Surprisingly, there is an explicit formula for $P_M$ in terms of $F_M$.   See \Thm{PMFMqt}.     

\Thm{SUFFthm}, the central result of this paper, contains a sufficient condition in terms of the function $P_M$ for $M$, to be \cmqd.   Actually, the implication (a)$\lr$(d) of \Thm{EQUIVthm} shows that this condition is also necessary.

\Sec{RLTsec} contains the definitions and basic properties of the three quantities $\Lt$, $\qt$ and $\rt$ (defined on $B^*$ rather than on $B$) analogous to those introduced in \Sec{RLsec}.  

The main results in \Sec{Fsec} are \Thm{Signsthm} and \Cor{Signscor}, results about $P_M$ and $F_M$, in which the concept of {\em exact} inf--convolution is important.   \Cor{Signscor}(b) will be used in the latter part of \Sec{NIsec} in our generalization of Rockafellar's surjectivity theorem to nonreflexive spaces. In \Cor{Signscor}(c), it is shown that if $F_M \ge 0$ on $B$ then $M$ is \cmqd.   Actually, this converse of this is true --- see the implication (a)$\lr$(c) of \Thm{EQUIVthm}.

In \Sec{Gsec}, we explore the elementary properties of the function $G_M$, and define the {\em Gossez extension}, $M^\sharp$, of $M$, (see \Def{MGdef}).  \big(We write $M^\sharp$ since the usual notation, $\overline M$, can have other interpretations.\big)    

This leads rapidly to \Sec{NIsec}, in which we define maximally monotone sets of {\em type (NI)}, originally introduced in \cite[1996]{RANGE}. The definition in condensed form is  ``$G_M \ge 0\ \on\ B^*$'' \big(see \eqref{NIdef2}: it also appears in expanded form in \eqref{NIdef1}\big).    It is proved in \Thm{AFMAXthm} that if $M$ is of type (NI) then $M^\sharp$ is a maximally monotone subset of $B^*$.  We show in \Ex{TAILNIex} that the tail operator, $T$, is maximally monotone but not of type (NI) and we exhibit in \Ex{TAILGex} a maximally monotone set of type (NI) whose Gossez extension is the graph of $T$. This leads to a negative answer to the question posed in \cite[Problem 12.7, p.\ 1047]{PARTONE}.
   \Thm{AFRLthm} is an existence theorem for maximally monotone subsets of $B$ of type (NI), which relies on the {\em exactness} of the inf--convolution already mentioned in the discussion of \Thm{Signsthm}.  Its consequence, \Cor{NISURcor}, is a generalization of Rockafellar's surjectivity theorem from reflexive Banach spaces to general Banach spaces (see \Cor{ROCKSURcor}).   Of course, Rockafellar's result is, in turn, a generalization of Minty's theorem from Hilbert spaces to reflexive Banach spaces.              

In \Sec{LINsec}, we make a short digression and discuss monotone linear\break subspaces of $B$ and their adjoint spaces.  \Lem{FNLlem} is about the {\em automatic maximality} of certain monotone linear subspaces.  \Thm{qqthm} contains a generalization of Brezis-Browder \cite[Theorem 2]{BB} to nonreflexive spaces that was first proved in \cite[Theorem 4.1, p.\ 4960]{BBWYBB}.   The proof given here is considerably shorter.

We collect in \Thm{EQUIVthm} the four criteria that we have already discussed for $M$ to be a \cmqd\ subset of $B$.   There are many other equivalent criteria  \big(see\break \cite[Introduction, pp.\ 6--7]{PARTTWO}\big).   \Thm{EQthm} contains nine statements equivalent to the statement that an element $b^*$ of $B^*$ belong to the Gossez extension of $M$.

The main result of \Sec{NETsec} is \Thm{NETthm}.   We assume in this that $f_0$ is a proper, convex lower semicontinuous function on $B$, $(z^*,z\dbs) \in B^*$, and we establish the existence of a net $(x_\alpha,{x_\alpha}^*)$ of elements of $B$ such that $\limsupn_\alpha f_0(x_\alpha,{x_\alpha}^*) \le {f_0}\dbs(z\dbs,\wh{z^*})$ and various other conditions are satisfied.   It is important to note that {\em monotonicity} is not mentioned at all in \Sec{NETsec}.  The main tool here is the formula for the biconjugate of a the maximum of a finite number of convex functions. 

We apply \Thm{NETthm} in \Thm{Dthm}, in which the we show the equivalence of $M$ being a \cmqd\ subset of $B$ and $M$ being of ``Type (D)'' or ``Type (WD)'' (see \Def{Ddef}), thus adding two additional criteria to \Thm{EQUIVthm}.

We do not know whether the techniques introduced in this paper shed any light on the {\em sum theorems for quasidense functions} (see \cite[Theorem 8.4, pp.\ 1036--7]{PARTONE} and \cite[Theorem 8.8, p\ 1039]{PARTONE} or the equivalence between quasidensity and {\em Type (FP)} \cite[Theorem 10.3, p.\ 21]{PARTTWO} -- related papers by other authors are dicussed in \cite[Remark 10.4, pp.\ 21--22]{PARTTWO}.    
\section{Banach space notation}\label{SNsec}
\begin{definition}\label{PCdef} 
If $X$ is a nonzero real Banach space and $f\colon\ X \to \rbar$, we write $\dom\,f := f^{-1}\RR$, and say that $f$ is {\em proper} if $\dom\,f \ne \emptyset$. If $f\colon X \to \rbar$ is proper and $g\colon X \to \RR$, we define $f\ssquare g\colon X \to [-\infty,\infty[$ by $(f\ssquare g)(x) := \infn_{y \in X}\big[f(y) + g(x - y)\big]$.  $f\ssquare g$ is the {\em inf-convolution} of $f$ and $g$.      We say that this inf-convolution is {\em exact} is the infimum is attained.    In this case we write $f\squaree g$.   So $(f\squaree g)(x) = \minn_{y \in X}\big[f(y) + g(x - y)\big]$.   This concept will be used in \Thm{Signsthm}, \Cor{Signscor}(b) and \Thm{AFRLthm} and its consequences.

We write $\PC(X)$ for the set of all proper convex functions from $X$ into $\rbar$ and $\PCLSC(X)$ for the set of all proper convex lower semicontinuous functions from $X$ into $\rbar$.   We write $X^*$ for the dual space of $X$ \big(with the pairing $\bra\cdot\cdot\colon X \times X^* \to \RR$\big).  If $f \in \PC(X)$ then, as usual, we define the {\em Fenchel conjugate}, $f^*$, of $f$ to be the function on $X^*$ given by
\begin{equation}\label{FSTAR}
x^* \mapsto \supn_X\big[x^* - f\big]\qquad(x^* \in X^*).
\end{equation}
We write $X\dbs$ for the bidual of $X$ \big(with the pairing $\bra\cdot\cdot\colon X^* \times X\dbs \to \RR$\big).\break   If $f \in \PC(X)$ and $f^* \in \PCLSC(X^*)$, we define $f\dbs\colon X\dbs \to \rbar$ by $f\dbs(x\dbs) := \sup_{X^*}\big[x\dbs - f^*\big]$.   If $x \in X$, we write $\wh x$ for the canonical image of $x$ in $X\dbs$, that is to say\quad $(x,x^*) \in X \times X^* \lr \Bra{x^*}{\wh x} = \bra{x}{x^*}$.   If $Y \subset X$ and $x \in X$, then $\I_{Y}(x) := 0$ if $x \in Y$ and $\I_{Y}(x) := \infty$ if $x \in X \setminus Y$. $\I_Y$ is the {\em indicator function of} $Y$.
\end{definition}
The basic result about $f\dbs$ is the {\em Fenchel--Moreau theorem}:
\begin{theorem}\label{FMthm}
Let $f \in \PCLSC(X)$ and $x \in X$.   Then $f\dbs(\wh x) = f(x)$, that is to say, $f(x) = \supn_{x^* \in X^*}\big[\bra{x}{x^*} - f^*(x^*)\big]$.
\end{theorem}
\begin{proof}
See Moreau, \cite[Section 5--6, pp.\ 26--39]{MOREAU} or Z\u{a}linescu, \cite[Theorem 2.3.3 pp.\ 77--78]{ZBOOK}.   This result is extended in \cite[Section 12, pp.\ 58--60]{HBM}. 
\end{proof}
\section{$q$, $L$ and $r$}\label{RLsec}
We define the norm on $B = E \times E^*$ by $\|(x,x^*)\| := \sqrt{\|x\|^2 + \|x^*\|^2}$,\break and represent $B^*$ by $E^* \times E\dbs$ under the pairing
$$\Bra{(x,x^*)}{(z^*,z\dbs)} := \bra{x}{z^*} + \bra{x^*}{z\dbs}.$$
Our discussions of monotonicity lead us to expressions of the form $\bra{x - y}{x^* - y^*}$.  We now introduce notations that will simplify this and similar expressions.   We have the identity $\bra{x - y}{x^* - y^*} = \bra{x}{x^*} - \bra{x}{y^*} - \bra{y}{x^*} + \bra{y}{y^*}$.   We define the quadratic function $q$ on $B$ by $q(x,x^*) := \bra {x}{x^*}$.   Thus, if $b = (x,x^*)$ and $c = (y,y^*)$, we have $q(b - c) = q(b) - \bra{x}{y^*} - \bra{y}{x^*} + q(c)$.   We now deal with the middle two terms of this.   Since  $\bra{x}{y^*} + \bra{y}{x^*} = \Bra{(x,x^*)}{(y^*,\wh y)}$, we define 
the linear isometry $L\colon\ B \to B^*$ by $L(y,y^*) := (y^*,\wh{y})$, and obtain the relationship $q(b - c) = q(b) - \bra{b}{Lc} + q(c)$.   It is sometimes useful to multiply $b$ and $c$ by scalars, and we end up with: for all $\lambda, \mu \in \RR$ and $d,e \in B$,
\begin{equation}\label{not1}
q(\lambda d - \mu e) = \lambda^2q(d) -  \lambda\mu\bra{d}{Le} + \mu^2q(e).
\end{equation}
A word is in order about the map $L$.   If $E$ is any Banach space then\quad $\wh{}$\quad is the {\em canonical map} from $E$ into $E\dbs$.   If $B$ is a Banach space of the special form $E \times E^*$ then we can think of $L$ as the {\em canonical map} from $B$ into $B^*$.
\par 
The following {\em weighted parallelogram law} is an immediate consequence of this: if $\lambda + \mu = 1$ and $d,e \in B$ then 
\begin{equation}\label{not2}
q(\lambda d + \mu e) + \lambda\mu q(d - e) = \lambda q(d) + \mu q(e).
\end{equation}
If we take $\lambda =\mu = \half$ and multiply by 4, we obtain the {\em usual parallelogram law}:
\begin{equation}\label{not2U}
q(d + e) + q(d - e) = 2q(d) + 2q(e).
\end{equation}
We define the quadratic function $r$ on $B$ by
\begin{equation}\label{rdef}
r := q + \half\|\cdot\|^2.
\end{equation}
The function $r$ appears in the reflexive case in \cite[Theorem~10.3,\ p.\ 36]{MANDM}  and  in Simons--Z\u{a}linescu \cite{SZNZ}, with the symbol ``$G$''.   It was used in the nonreflexive case by Zagrodny in \cite{ZAGRODNY}.   Since $\|L\| \le 1$, for all $b \in B$, $|q(b)| = \half|\bra{b}{Lb}| \le \half\|b\|\|Lb\| \le \half\|b\|^2$, consequently
\begin{equation}\label{rpos}
0 \le  r \le \|\cdot\|^2\ \on\ B.
\end{equation}
Now let $d,e \in B$.   Since $\|L\| \le 1$, from \eqref{not1} with $\lambda = 1$ and $\mu = -1$  
\begin{equation}\label{not8}
q(d + e) = q(d) + \bra{d}{Le} + q(e) \le q(d) + \|d\|\|e\| + q(e).
\end{equation}
Since $\half\|d + e\|^2 \le \half(\|d\| + \|e\|)^2 \le \half\|d\|^2 + \|d\|\|e\| + \half\|e\|^2$, from \eqref{rdef} and \eqref{rpos}, 
\begin{equation}\label{not8R}
r(d + e) \le r(d) + 2\|d\|\|e\| + r(e) \le \|d\|^2 + 2\|d\|\|e\| + r(e).
\end{equation}
\par
It is easy to see that, for all $b,c \in B$, $\bra{b}{Lc} = \bra{c}{Lb}$ and, for all $b \in B$, $\half\bra{b}{Lb} = q(b)$.   It follows that, for all $b,c \in B$,
\begin{equation*}
|q(b) - q(c)| = \half|\bra{b}{Lb} - \bra{c}{Lc}| = \half|\bra{b - c}{L(b + c)}| \le \half\|b + c\|\|b - c\|.
\end{equation*}
Consequently,
\begin{equation}\label{qcont}
q\hbox{ is continuous}.
\end{equation}
\section{Monotone, Quasidense and \cmqd\ sets}\label{QDsec}
Let $A \subset B$.   Obviously $A$ is {\em monotone} in the usual sense exactly when, for all $d,e \in A$, $q(d - e) \ge 0$.
\begin{definition}\label{QDdef}
Let $A \subset B$.   We say that $A$ is {\em quasidense in }$B$ if,
\begin{equation}\label{QD1}
\all\ b \in B,\quad \infn r(A - b) = 0.
\end{equation} 
Equivalently, for all $(x,x^*) \in B$, 
\begin{equation}\label{QD2}
\infn_{(s,s^*) \in A}\big[\half\|s - x\|^2 + \half\|s^* - x^*\|^2 + \bra{s - x}{s^* - x^*}\big] = 0.
\end{equation}
We shall write {\em ``\cmqd''} for ``closed, monotone and quasidense''. 
\end{definition}
\begin{theorem}[Quasidensity and maximality]\label{RLMAXthm}
Let $A$ be a \cmqd\ subset of $B$.   Then $A$ is maximally monotone. 
\end{theorem}
\begin{proof}
Let $b \in B$ and $A \cup \{b\}$ be monotone.   Let $\eps > 0$.   By hypothesis, there exists $a \in A$ such that $q(a - b) + \half\|a - b\|^2 < \eps$.   Since $A \cup \{b\}$ is monotone, $q(a - b) \ge 0$, and so $\half\|a - b\|^2 < \eps$.   However, $A$ is closed.   Thus, letting $\eps \to 0$,  $b \in A$.   This completes the proof of the maximality.   \big(This proof is adapted from that of \cite[Lemma 4.7, p.\ 1027]{PARTONE} --- the result appears explicitly in \cite[Theorem 7.4(a), pp.\ 1032--1033]{PARTONE}.\big)     
\end{proof}
\begin{remark}\label{TAILrem}
The {\em tail operator} is the linear operator $T\colon\ \ell_1 \to \ell_\infty$ defined by $(Tx)_n = \sum_{k \ge n} x_k$.   It was proved in \cite[Example 7.10, p.\ 1034--1035]{PARTONE} that $G(T)$ is maximally monotone but not quasidense.  This proof should be compared to the proof of the allied result that appears in \Ex{TAILNIex}.   Actually, we shall see in (a)$\iff$(b) of \Thm{EQUIVthm} that these two results are equivalent.
\end{remark}
\section{The function $P_M\colon B \to [0,\infty]$}\label{Psec}
We suppose for the rest of  this paper that $M$ is a maximally monotone subset of $B$.   The main result of this section is \Thm{SUFFthm}, in which we give a sufficient condition for $M$ to be quasidense. 
\par
Let $b \in B$.   If $\inf q(M - b) \ge 0$ then $M\cup\{b\}$ is monotone so, from the maximality, $b \in M$, from which $\inf q(M - b) \le q(b - b) = 0$.   Consequently, $\inf q(M - b) \le 0$, with equality exactly when $b \in M$.  We now define
\begin{equation}\label{P}
P_M(b) = -\inf q(M - b) \in [0,\infty].
\end{equation}
It is clear from the continuity of $q$ \big(see \eqref{qcont}\big) that
\begin{equation}\label{P2}
P_M\hbox{ is lower semicontinuous\enspace and\enspace}M = {P_M}^{-1}\{0\} = {P_M}^{-1}\big(\,]-\infty,0\,]\big).
\end{equation}
We collect together in \Lem{Plem} the elementary properties of the function $P_M$.   (a) and (b)  will be used in \Lems{FMLPMq} and \ref{PAIRSlem}, respectively, and (c) and (d) will be used in \Thm{SUFFthm}. 
\begin{lemma}\label{Plem}
{\rm(a)}\enspace We have $P_M \ssquare q = 0\ \on\ M$.\par
{\rm(b)}\enspace Let $b \in B$ and $m \in M$.   Then $\half\|m - b\|^2 \le r(m - b) + P_M(b)$.\par
{\rm(c)}\enspace Let $a,b \in B$.  Then $-q(b - a) \le 2P_M(a) + 2P_M(b)$.
\par
{\rm(d)}\enspace Let $d,e,t \in B$.   Then    
$$\big(\|e\| - \|t\|\big)^2 \le 4P_M(d -e) + 2r(e) + 4P_M(d -t) + 2r(t).$$   
\end{lemma} 
\begin{proof}
(a)\enspace Let $m \in M$ and $b \in B$.   Since $\inf q(M - b) = -P_M(b)$,\break $q(m - b) \ge -P_M(b)$.   Thus $P_M(b) + q(m - b) \ge 0$.   Taking the infimum over $b$, $(P_M \ssquare q)(m) \ge 0$.   On the other hand, from \eqref{P2}, $$(P_M \ssquare q)(m) \le P_M(m) + q(m - m) = 0.$$ 
This completes the proof of (a).
\par
(b)\enspace By virtue of \eqref{rdef}, $\half\|m - b\|^2 = r(m - b) -q(m - b)$ and, by virtue of  \eqref{P}, $-q(m - b) \le P_M(b)$.   This completes the proof of (b).
\par
(c)\enspace From \eqref{P}, and the parallelogram law, \eqref{not2U}, with $d = m - a$ and\break $e = m - b$,
\begin{align*}
- q(b - a) &\le 4P_M(\half a + \half b) - q(b - a)
= -4\infn_{m \in M}q(m - \half a - \half b) - q(b - a)\\
&= -\infn_{m \in M}q(2m - a - b) - q(b - a)\\
&= -\infn_{m \in M}\big[q(2m - a - b) + q(b - a)\big]\\
&= -\infn_{m \in M}\big[2q(m - a) + 2q(m - b)\big]\\
&\le -2\infn_{m \in M}q(m - a) - 2\infn_{m \in M}q(m - b) = 2P_M(a) + 2P_M(b). 
\end{align*}
This completes the proof of (c).
\par
(d)\enspace From the fact that $\|L\| \le 1$, and  (c) with $a = d - e$ and $b = d - t$, 
\begin{align*}
\big(\|e\| - \|t\|\big)^2 &= \|e\|^2 + \|t\|^2 - 2\|e\|\|t\| \le \|e\|^2 + \|t\|^2 + 2\bra{e}{Lt}\\
&= \|e\|^2 + \|t\|^2 - 2q(e - t) + 2q(e) + 2q(t)\\ 
&\le 4P_M(d -e) + 2q(e) + \|e\|^2 + 4P_M(d -t) + 2q(t) + \|t\|^2.
\end{align*}
\eqref{rdef} gives\quad $2q(e) + \|e\|^2 = 2r(e)$\quad and\quad $2q(t) + \|t\|^2 = 2r(t)$,\quad which completes the proof of (d).
\end{proof}
\begin{theorem}\label{SUFFthm}
Suppose that
\begin{equation}\label{SUFF1}
P_M \ssquare r \le 0\ on\ B.
\end{equation}  
\par
\noindent
{\rm(a)}\enspace Let $c \in B$. Then $\inf\|M - c\|^2 \le 4P_M(c)$.  
\par
\noindent
{\rm(b)}\enspace The set $M$ is \cmqd.
\end{theorem}
\begin{proof}(a) We can clearly suppose that  $P_M(c) < \infty$.   Let $P_M(c) < \Lambda < \infty$.   Let $\eta \in\ ]0,1[$ be chosen so small that
$$4\tsumn_{n = 1}^\infty \eta^n + 2\sqrt{P_M(c)} < 2\sqrt\Lambda.$$
If $b_1, \dots, b_{n - 1}\in B$ then, from \eqref{SUFF1}, $(P_M \ssquare r)(c + \tsum_{k = 1}^{n-1} b_k) \le 0$, and so, for all $n \ge 1$, there exists $b_n \in B$ such that $P_M(c + \tsum_{k = 1}^n b_k) + r(-b_n) < \eta^{2n + 2}$.   Since $P_M, r \ge 0$ on $B$ and $r$ is even, we can choose $b_1,b_2\dots \in B$ inductively so that, for all $n \ge 1$,
\begin{equation}\label{SUFF2}
P_M(c + \tsum_{k = 1}^n b_k) \vee r(b_n) = P_M(c + \tsum_{k = 1}^n b_k) \vee r(-b_n) < \eta^{2n + 2}.
\end{equation}
From \eqref{rdef}, \eqref{SUFF2} and \Lem{Plem}(c) with $a = c + \tsum_{k = 1}^{n - 1} b_k$ and $b = c + \tsum_{k = 1}^{n} b_k$, for all $n \ge 1$,
\begin{align*}
\half\|b_n\|^2 &= r(b_n) - q(b_n)
\le \eta^{2n + 2} + 2P_M(c + \tsum_{k = 1}^{n - 1} b_k) + 2P_M(c + \tsum_{k = 1}^n b_k)\\
&\le \eta^{2n + 2} + 2P_M(c + \tsum_{k = 1}^{n - 1} b_k) + 2\eta^{2n + 2} = 3\eta^{2n + 2} + 2P_M(c + \tsum_{k = 1}^{n - 1} b_k).
\end{align*}
When $n = 1$, this gives us $\half\|b_{1}\|^2 \le  3\eta^{4} + 2P_M(c)$.   Multiplying this by 2,
$\|b_1\|^2 \le  6\eta^{4} + 4P_M(c) \le  16\eta^{2} + 4P_M(c)$ and so, in this case,
\begin{equation}\label{SUFF8}
\|b_{1}\| \le  4\eta + 2\sqrt{P_M(c)}.
\end{equation}
On the other hand, when $n \ge 2$, this gives us $\half\|b_{n}\|^2 \le  3\eta^{2n + 2} + 2\eta^{2n} < 5\eta^{2n}$, from which\quad $\|b_n\| \le 4\eta^{n}$.\quad  It follows that $\tsum_{n = 1}^\infty b_n$ exists.   Let\break  $m := c + \sum_{n = 1}^\infty b_n$.  From \eqref{P2} and \eqref{SUFF2}, $P_M(m) \le 0$, and so $m \in M$.  From \eqref{SUFF8},
\begin{align*}
\|m - c\| &= \|\tsum_{n = 1}^\infty b_n\| \le \|b_1\| + \tsum_{n = 2}^\infty\|b_n\|  \le  4\eta + 2\sqrt{P_M(c)} + 4\tsum_{n = 2}^\infty\eta^n\\
&= 4\tsum_{n = 1}^\infty \eta^n + 2\sqrt{P_M(c)} < 2\sqrt\Lambda.
\end{align*}
Thus $\|m - c\|^2 < 4\Lambda$, and (a) follows by letting $\Lambda \to P_M(c)$.
\par
(b)\enspace Let $d \in B$ and $\eps > 0$.   Since $(P_M \ssquare r)(d) \le 0$, there exists $t \in B$ such that $P_M(d - t) + r(t) < \hexth$.   Since $P_M, r \ge 0$ on $B$, 
\begin{equation}\label{SUFF3}
P_M(d - t) \vee r(t) < \hexth.
\end{equation}
Let $0 < \delta < \half$ and
\begin{equation}\label{SUFF4}
2\delta(\delta + \|t\| + 1) < \eps.
\end{equation}
Similarly, there exists $e \in B$, depending on $\delta$, such that
\begin{equation}\label{SUFF5}
P_M(d - e) \vee r(e) < \fourth\delta^2 < \hexth.
\end{equation}
From \eqref{SUFF3}, \eqref{SUFF5} and \Lem{Plem}(d), $\big(\|e\| - \|t\|\big)^2 \le \fourth + \eighth + \fourth + \eighth < 1$ and so 
\begin{equation}\label{SUFF7}
\|e\| < \|t\| + 1.
\end{equation}
From (a), with $c := d - e$, and \eqref{SUFF5} we can choose $m \in M$ such that
\begin{equation}\label{SUFF6}
\|m - d + e\| <  \delta. 
\end{equation}
From \eqref{not8R} with $d := m - d + e$, \eqref{SUFF4}--\eqref{SUFF6}, and the  fact that $r$ is even, 
\begin{align*}
r(m - d) &= r\big([m - d + e] - e\big) \le \|m - d + e\|^2 + 2\|m - d + e\|\|e\| + r(e)\\
&\le \delta^2 + 2\delta\|t\| + 2\delta + \fourth\delta^2 \le 2\delta(\delta + \|t\| + 1) < \eps.
\end{align*}
Since this holds for an arbitrary $d \in B$ and $\eps > 0$, $M$ is quasidense in $B$.   $M$ is obviously closed.   Thus $M$ is \cmqd.
\end{proof}
\section{$\qt$, $\Lt$ and $\rt$}\label{RLTsec}
The norm on $B^* = E^* \times E\dbs$ is given by  $\|(y^*,y\dbs)\| := \sqrt{\|y^*\|^2 + \|y\dbs\|^2}$. 
By analogy with \Sec{RLsec}, we define the continuous quadratic function $\qt$ on $B^*$ by $\qt(y^*,y\dbs) := \bra {y^*}{y\dbs}$, and the linear isometry $\Lt\colon\ B^* \to B\dbs$ by $\Lt(y^*,y\dbs) = \big(y\dbs,\wh{y^*}\big)$.      Then, for all $d^*,e^*\in B^*$, $\Bra{d^*}{\Lt e^*} = \Bra{e^*}{\Lt d^*}$ and, for all $d^* \in B^*$, $\half\Bra{d^*}{\Lt d^*} = \qt(d^*)$.   By analogy with \eqref{not1}, for all $\lambda,\mu \in \RR$ and $d^*,e^* \in B^*$,
\begin{equation}\label{not3}
\qt(\lambda d^* - \mu e^*) = \lambda^2\qt(d^*) - \lambda\mu \Bra{d^*}{\Lt e^*} + \mu^2\qt(e^*).
\end{equation}
By analogy with \eqref{not2}, if $\lambda + \mu = 1$ and  $b^*,c^* \in B^*$ then we have the {\em weighted parallelogram law}
\begin{equation}\label{not10}
\qt(\lambda b^* + \mu c^*) + \lambda\mu \qt(b^* - c^*) = \lambda \qt(b^*) + \mu \qt(c^*).
\end{equation}
Let $b \in B$.   Replacing $b^*$ and $c^*$ in \eqref{not10} by $Lb - d^*$ and $Lb - e^*$, respectively, 
\begin{equation}\label{not4}
\qt(Lb - \lambda d^* - \mu e^*) + \lambda\mu\qt(e^* - d^*) = \lambda\qt(Lb - d^*) + \mu\qt(Lb - e^*).
\end{equation}
It is easy to see that, for all $b \in B$ and $b^* \in B^*$,
\begin{equation}\label{not6}
\Lt L(b) = \wh b\quand\Bra{Lb}{\Lt b^*} =  \bra{b}{b^*}.
\end{equation}
We define the quadratic function $\rt$ on $B^*$ by 
\begin{equation}\label{rtdef}
\rt := \qt+ \half\|\cdot\|^2, \hbox{ and so, as in \eqref{rpos}, } 0 \le \rt \le \|\cdot\|^2\ \on\ B^*.
\end{equation}
We note then that
\begin{equation}\label{not7}
\qt \circ L = q \quand \rt\circ L = r.
\end{equation}
From \eqref{not3} with $\lambda = \mu = \half$, \eqref{not7} and \eqref{not6}, for all $b \in B$ and $b^* \in B^*$, 
\begin{equation}\label{not9}
\qt(Lb - b^*) = q(b) - \bra{b}{b^*} + \qt(b^*).
\end{equation}
\section{The function $F_M\colon B^* \to\ ]-\infty,\infty]$}\label{Fsec}
\begin{definition}\label{Fdef}
Let $M$ be a maximally monotone subset of $B$.   If $b^* \in B^*$, we define
\begin{equation}\label{DEL}
 F_M(b^*):=  -\infn_{b \in B}\big[P_M(b) + \qt(b^* - Lb)\big] \in \,]-\infty,\infty].
\end{equation}
See \Thm{PMFMqt} for an unexpected explicit formula for $P_M$ in terms of $F_M$.  The other significant results in this section are \Thm{Signsthm} and \Cor{Signscor}.  
\end{definition}
\begin{lemma}\label{FMLPMq}
$F_M \circ L = -(P_M\ssquare q)$ on $B$ and $F_M \circ L =  0$ on $M$.  
\end{lemma}
\begin{proof}
If $c \in B$ then, from \eqref{DEL},\quad $F_M(Lc) =  -\infn_{b \in B}\big[P_M(b) + \qt(Lc - Lb)\big]$.\quad   Thus, from \eqref{not7}, $F_M(Lc) = -\infn_{b \in B}\big[P_M(b) + q(c - b)\big]
= -(P_M\ssquare q)(c)$.   The second observation follows from the first and \Lem{Plem}(a). 
\end{proof}
%
\begin{definition}\label{PHIdef}
Let $\Phi_M := P_M + q$ on $B$.  From \eqref{P}, for all $b \in B$, we have $\Phi_M(b)\in [q(b),\infty]$.   See \Rem{PHrem} for more on this function.
\end{definition}
\begin{lemma}\label{Flem}
{\rm(a)}\enspace $\Phi_M \in \PCLSC(B)$.\par\noindent
{\rm(b)}\enspace For all $b^* \in B^*$, ${\Phi_M}^*(b^*) = F_M(b^*) + \qt(b^*)$.\par\noindent
{\rm(c)}\enspace  For all $b^* \in B^*$, ${\Phi_M}\dbs\big(\Lt b^*\big) = \qt(b^*) -(F_M \ssquare \qt)(b^*)$.
\end{lemma}
\begin{proof}
From \eqref{P} and \eqref{not8}, for all $b \in B$, 
\begin{align*}
\Phi_M(b) &= P_M(b) + q(b) = -\infn_{m \in M} q(m - b) + q(b)\\
&= -\infn_{m \in M}\big[q(m) - \bra{b}{Lm} + q(b)\big] + q(b)\\
&=  \supn_{m \in M}\big[\bra{b}{Lm} - q(m)\big],
\end{align*}
from which (a) follows immediately.   From \Def{PHIdef}, \eqref{not9} and \eqref{DEL},  
\begin{align*}
{\Phi_M}^*(b^*) &= \supn_{b \in B}\big[\bra{b}{b^*} - \Phi_M(b)\big]
= \supn_{b \in B}\big[\bra{b}{b^*} - P_M(b) - q(b)]\\
&= \supn_{b \in B}\big[- P_M(b) -  \qt(b^* - Lb) + \qt(b^*)\big]\\
&= -\infn_{b \in B}\big[P_M(b) + \qt(b^* - Lb)\big] + \qt(b^*) = F_M(b^*) + \qt(b^*).
\end{align*}
This completes the proof of (b).
Let $b^* \in B^*$.  From (b) and \eqref{not3},
\begin{align*}
{\Phi_M}\dbs\big(\Lt b^*\big) -  \qt(b^*) 
&= \supn_{c^* \in B^*}\big[\bra{c^*}{\Lt b^*} - {\Phi_M}^*(c^*)\big] -  \qt(b^*)\\
&= \supn_{c^* \in B^*}\big[\bra{c^*}{\Lt b^*} - F_M(c^*) - \qt(c^*) - \qt(b^*)\big]\\
&= -\infn_{c^* \in B^*}\big[F_M(c^*) + \qt(c^*) -\bra{c^*}{\Lt b^*} + \qt(b^*)\big]\\
&= -\infn_{c^* \in B^*}\big[F_M(c^*) + \qt(b^* - c^*)\big] = - (F_M \ssquare \qt)(b^*).
\end{align*}
This completes the proof of (c).
\end{proof}
%
\begin{theorem}\label{PMFMqt}
$P_M = - (F_M \ssquare \qt) \circ L$ on $B$.
\end{theorem}
\begin{proof}Let $b \in B$.   Then, from \Def{PHIdef}, \Lem{Flem}(b), the Fenchel--Moreau Theorem (\Thm{FMthm}) and \eqref{not9},  
\begin{align*}
P_M(b) &= \Phi_M(b) - q(b) = \supn_{b^* \in B^*}\big[\bra{b}{b^*} - {\Phi_M}^*(b^*) - q(b)\big]\\
&= \supn_{b^* \in B^*}\big[\bra{b}{b^*} - F_M(b^*) -\qt(b^*) - q(b)\big]\\
&= -\infn_{b^* \in B^*}\big[F_M(b^*) + q(b) - \bra{b}{b^*} + \qt(b^*)\big]\\
&= -\infn_{b^* \in B^*}\big[F_M(b^*) + \qt(Lb - b^*)\big] = - (F_M \ssquare \qt)(Lb).
\end{align*}
This completes the proof of \Thm{PMFMqt}. 
\end{proof}
\begin{lemma}\label{JClem}
Let $c \in B$, and define the real, convex continuous function $j_c$ on $B$ by ${j_c}(b) := \half\|c - b\|^2$.  Let $a^* \in B^*$.   Then ${j_c}^*(a^*) = \half\|a^*\|^2 + \bra{c}{a^*}$. 
\end{lemma}
\begin{proof}${j_c}^*(a^*) = \supn_{b \in B}\big[\bra{b}{a^*} - {j_c}(b)\big] = \supn_{b \in B}\big[\bra{c - b}{a^*} - {j_c}(c - b)\big]
= \supn_{b \in B}\big[\bra{b}{-a^*} - \half\|b\|^2\big] + \bra{c}{a^*} = \half\|-a^*\|^2 + \bra{c}{a^*}$.   This completes the proof of \Lem{JClem}.
\end{proof}
\begin{theorem}[Result with exactness]\label{Signsthm}
$P_M \ssquare r =  - (F_M \squaree \rt)\circ L$ on $B$.
\end{theorem}
\begin{proof}Let $c \in B$.  From \Def{PHIdef}, \eqref{rdef} and \eqref{not8},
\begin{align*}
(P_M \ssquare r)(c) &= \infn_{b \in B}\big[\Phi_M(b) - q(b) + q(c - b) + \half\|c - b\|^2\big]\\
&= \infn_{b \in B}\big[\Phi_M(b) + q(c) - \bra{b}{Lc} + {j_c}(b)\big]\\
&= q(c) - \supn_{b \in B}\big[\bra{b}{Lc} - (\Phi_M + {j_c})(b)\big]\\
&= q(c) - (\Phi_M + {j_c})^*(Lc).
\end{align*}
Using Rockafellar's version of the Fenchel duality theorem \big(Rockafellar,\break \cite[Theorem~3(a), p.\ 85]{FENCHEL}, Z\u alinescu, \cite[Theorem~2.8.7(iii), p.\ 127]{ZBOOK}, or\break \cite[Corollary~10.3, p.\ 52]{HBM}\big), \Lem{JClem} and the fact that $\bra{c}{Lc} = 2q(c)$,
\begin{align*}
(P_M &\ssquare r)(c) = q(c) - ({\Phi_M}^* \squaree {j_c}^*)(Lc)\\
&= q(c) - \minn_{b^* \in B^*}\big[{\Phi_M}^*(b^*) + {j_c}^*(Lc - b^*)\big]\\
&= q(c) - \minn_{b^* \in B^*}\big[{\Phi_M}^*(b^*) + \half\|Lc - b^*\|^2 + \bra{c}{Lc - b^*}\big]\\
&=  - \minn_{b^* \in B^*}\big[{\Phi_M}^*(b^*)  + \half\|Lc - b^*\|^2 + q(c) -\bra{c}{b^*}\big].
\end{align*}
From \eqref{not9}, $q(c) - \bra{c}{b^*} = \qt(Lc - b^*) - \qt(b^*)$.   
Hence, from \eqref{rtdef},
\begin{align*}
(P_M \ssquare r)(c) &=  - \minn_{b^* \in B^*}\big[{\Phi_M}^*(b^*)  + \half\|Lc - b^*\|^2 + \qt(Lc - b^*) - \qt(b^*)\big]\\
&=  - \minn_{b^* \in B^*}\big[{\Phi_M}^*(b^*) - \qt(b^*)  + \rt(Lc - b^*)\big].
\end{align*}
Consequently, from \Lem{Flem}(b),
\begin{align*}
(P_M \ssquare r)(c) &=  - \minn_{b^* \in B^*}\big[F_M(b^*) + \rt(Lc - b^*)\big]  =  - (F_M \squaree \rt)(Lc).
\end{align*}
This completes the proof of \Thm{Signsthm}.
\end{proof}
\Cor{Signscor}(b) below will be used in \Thm{AFRLthm} and \Cor{NISURcor}, a generalization of Rockafellar's surjectivity theorem to nonreflexive spaces.   It will also be used in \Thm{EQUIVthm}.
\begin{corollary}\label{Signscor}
Let $F_M \ge 0$ on $B^*$.   Then:
\par
\noindent
{\rm(a)}\enspace $P_M \ssquare r = 0\ \on\ B$.
\par
\noindent
{\rm(b)}\enspace $(F_M \squaree \rt)\circ L = 0\ \on\ B$.
\par
\noindent
{\rm(c)}\enspace The set $M$ is \cmqd.
\end{corollary}
\begin{proof}
 From \eqref{P} and \eqref{rpos}, $P_M \ge 0$ on $B$ and $r \ge 0$ on $B$, and so $P_M \ssquare r \ge 0$ on $B$.  By hypothesis, $F_M \ge 0$ on $B^*$ and, from \eqref{rtdef}, $\rt \ge 0$ on $B^*$, consequently $(F_M \squaree \rt)\circ L \ge 0$ on $B$.   (a) and (b) now follow from \Thm{Signsthm}, and (c) is immediate from (a) and \Thm{SUFFthm}(b).
\end{proof}
\begin{remark}\label{PHrem}
In more conventional notation, $\Phi_M$ is given by:
\begin{equation*}
\Phi_M(x,x^*) := \supn_{(s,s^*) \in M}\big[\bra{s}{x^*} + \bra{x}{s^*} - \bra{s}{s^*}\big]\quad\big((x,x^*) \in E \times E^*\big).
\end{equation*}
$\Phi_M$ is the {\em Fitzpatrick function} of $M$, introduced in \cite{FITZ}.   See \cite[Section 23, pp. 99-103]{HBM} for a history of $\Phi_M$.   
\end{remark}
\section{The function $G_M\colon B^* \to\ ]-\infty,\infty]$}\label{Gsec}
\begin{definition}\label{Gdef}
Let $M$ be a maximally monotone subset of $B$.  If $b^* \in B^*$, we define
\begin{equation}\label{G}
G_M(b^*) := -\inf\qt(L(M) - b^*) \in \rbar.
\end{equation}
\end{definition}
In \Lem{Glem}, we discuss two elementary properties of the function $G_M$.  \Lem{Glem}(a) will be used in \Thm{NIrefl} and \Thm{EQthm}\break\big((d)$\lr$(e)\big); \Lem{Glem}(b) will be used in \Lem{GFlem}, \Thm{AFRLthm} and\break \Thm{EQUIVthm}\big((b)$\lr$(c)\big).  
\begin{lemma}\label{Glem}
{\rm(a)}\enspace     $G_M\circ L = P_M$ on $B$.
\par\noindent
{\rm(b)}\enspace $G_M \le F_M$ on $B^*$.
\end{lemma}
\begin{proof}
Let $b \in B$. From \eqref{G}, \eqref{not7} and \eqref{P}
$$-G_M(Lb) = \inf \qt\big(L(M) - Lb\big) = \inf q(M - b) = -P_M(b),$$
which gives (a).   Now let $b^* \in B^*$.   From \eqref{DEL}, \eqref{P2} and \eqref{G},
\begin{align*}
-F_M(b^*) &= \infn_{b \in B}\big[P_M(b) + \qt(b^* - Lb)\big] \le \infn_{m \in M}\big[P_M(m) + \qt(b^* - Lm)\big]\\
& =  \infn_{m \in M}\big[0 + \qt(b^* - Lm)\big] =  \infn_{m \in M}\qt(b^* - Lm) = -G_M(b^*),
\end{align*}
which gives (b).
\end{proof}
\begin{definition}\label{MGdef}
The {\em Gossez Extension of $M$} is defined as the set, $M^\sharp$, consisting of all those elements of $B^*$ that are monotonically related to the canonical image of $M$ in $B^*$, that is to say $\big\{b^* \in B^*\colon\ \all\ m \in M,\ \qt(b^* - Lm) \ge 0\big\}$.     See Phelps, \cite[Definition 3.1, pp.\ 216-217]{PRAGUE}.  From \eqref{G},
\begin{equation}\label{MG}
M^\sharp = \big\{b^* \in B^*\colon\ G_M(b^*) \le 0\big\}.
\end{equation}
We use the notation $M^\sharp$ because the more usual $\overline{M}$ has other interpretations.
\end{definition}
\begin{lemma}\label{LMMFlem}
We have
\begin{equation}\label{AFMAX1}
M = L^{-1}M^\sharp\ \big(\hbox{from which }L(M) \subset M^\sharp\big).
\end{equation}
\end{lemma}
\begin{proof}
Let $b \in B$.   It follows from \eqref{not7}, \eqref{P} and \eqref{P2} that $Lb \in M^\sharp \iff \inf_{m \in M}\qt(Lb - Lm) \ge 0 \iff \inf_{m \in M} q(b - m) \ge 0 \iff P_M(b) \le 0 \iff\break b \in M$.  This completes the proof of \Lem{LMMFlem}. 
\end{proof}

\section{Type (NI)}\label{NIsec}
%
\begin{definition}\label{NIdef}
Let $M$ be a maximally monotone subset of $B$.   We say that $M$ is of {\em of type (NI)} if
\begin{equation}\label{NIdef2}
G_M \ge 0\ \on\ B^*.
\end{equation}
In more conventional notation, 
\begin{equation}\label{NIdef1}
\all\ (y^*,y\dbs) \in E^* \times E\dbs,\quad\infn_{(x,x^*) \in M}\bra{x^* - y^*}{\wh x - y\dbs} \le 0.
\end{equation} 
See \cite[Definition 10, p.\ 183]{RANGE}. ``(NI)'' stands for ``negative infimum''.
\end{definition}
\begin{fact}\label{FOLKLORE}
Every monotone linear map with full domain from a Banach space into its dual is maximally monotone.   See Phelps--Simons, \cite[Corollary 2.6, p.\ 306]{PS}.   We do not know the original source of this result.
\end{fact}
\begin{example}[The tail operator, see \Rem{TAILrem}]\label{TAILNIex}
Let $E = \ell_1$, and define $T\colon\ \ell_1 \mapsto \ell_\infty = E^*$ by $(Tx)_n = \sum_{k \ge n} x_k$.   Then {\em $T$ is maximally monotone but not of type (NI).}    

\end{example}
\begin{proof}
The maximal monotonicity of $T$ follows from Fact \ref{FOLKLORE}.   We now let $y^* := (1,1,\dots) \in {\ell_1}^* = \ell_\infty$ and $y\dbs \in {\ell_\infty}^*$ be a Banach limit.   Let $x \in \ell_1$, and write $\sigma = \bra{x}{y^*} = \sum_{n \ge 1}x_n$.     Now 
\begin{align*}
\bra{x}{Tx} &= \ts\sum_{n \ge 1}x_n\sum_{k \ge n}x_k = \sum_{n \ge 1}x_n^2 + \sum_{n \ge 1}\sum_{k > n}x_nx_k\\
&\ge \ts\half\sum_{n \ge 1}x_n^2 + \sum_{n \ge 1}\sum_{k > n}x_nx_k = \half\sigma^2.
\end{align*}
We also have $\bra{Tx}{y\dbs} = \lim_{n \to \infty}(Tx)_n = 0$ and $\bra{y^*}{y\dbs} = \lim_{n \to \infty}1 = 1$.   Thus, by addition, $\bra{Tx - y^*}{\wh x - y\dbs} = \bra{x}{Tx} - \bra{x}{y^*} - \bra{Tx}{y\dbs} + \bra{y^*}{y\dbs}
\ge \half\sigma^2 - \sigma + 1 = \half(\sigma - 1)^2 + \half \ge \half$.   So\quad $\infn_{(x,x^*) \in G(T)}\bra{x^* - y^*}{\wh x - y\dbs} \ge \half$, and consequently \eqref{NIdef1} implies that $T$ is not of type (NI). 
\end{proof}
\Lem{GFlem} will be used in \Thm{AFMAXthm},  \Thm{EQthm}\big((b)$\lr$(c)\big) and\break \Thm{EQthm}\big((f)$\lr$(g)\big).
\begin{lemma}\label{GFlem}
Let $M$ be of type (NI), $d^*,e^* \in B^*$, $\lambda,\mu > 0$ and  $\lambda + \mu = 1$.    Then
\begin{equation}\label{G1}
\lambda G_M(d^*) + \mu G_M(e^*) + \lambda\mu\qt(e^* - d^*) \ge 0,
\end{equation}
\begin{equation}\label{F1}
\lambda F_M(d^*) + \mu F_M(e^*) + \lambda\mu\qt(e^* - d^*) \ge 0,
\end{equation}
\begin{equation}\label{FG2}
{G_M}(e^*) = 0 \qlr (G_M \ssquare \qt)(e^*) = 0
\end{equation}
and
\begin{equation}\label{FG3}
F_M(e^*) = 0 \qlr {\Phi_M}\dbs\big(\Lt e^*\big) = \qt(e^*).
\end{equation}        
\end{lemma}
\begin{proof}
Let $m \in M$.   From \eqref{not4} and \eqref{G},
\begin{align*}
\qt(Lm - \lambda d^* - \mu e^*) + \lambda\mu\qt(e^* - d^*) &= \lambda\qt(Lm - d^*) + \mu\qt(Lm - e^*)\\
&\ge -\lambda G_M(d^*) - \mu G_M(e^*).
\end{align*}
Thus, taking the infimum over $m$,
\begin{align*}
-G_M(\lambda d^* + \mu e^*) + \lambda\mu\qt(e^* - d^*) \ge -\lambda G_M(d^*) - \mu G_M(e^*).
\end{align*}
\eqref{G1} follows since \eqref{NIdef2} implies that $G_M(\lambda d^* + \mu e^*) \ge 0$, and \eqref{F1} follows from \eqref{G1} and \Lem{Glem}(b).
If ${G_M}(e^*) = 0$ then, from \eqref{G1},  for all $d^* \in B^*$ and  $\lambda,\mu > 0$ such that $\lambda + \mu = 1$,\quad $\lambda G_M(d^*) + \lambda\mu\qt(e^* - d^*) \ge 0$\quad from which\quad $G_M(d^*) + \mu\qt(e^* - d^*) \ge 0$.\quad Letting $\mu \to 1$,\quad $G_M(d^*) + \qt(e^* - d^*) \ge 0$.\quad   Taking the infimum over $d^* \in B^*$,\quad $(G_M \ssquare \qt)(e^*) \ge 0$.\quad  On the other hand, $(G_M \ssquare \qt)(e^*) \le G_M(e^*) + \qt(e^* - e^*) = 0$ and so $(G_M \ssquare \qt)(e^*) = 0$.   This\break completes the proof of \eqref{FG2}.   If $F_M(e^*) = 0$ then, arguing as above and using  \eqref{F1} instead of \eqref{G1}, $\big(F_M \ssquare \qt\big)(e^*) = 0$, and so \Lem{Flem}(c) implies that ${\Phi_M}\dbs\big(\Lt e^*\big) = \qt(e^*)$, which completes the proof of \eqref{FG3}.    
\end{proof}
\begin{definition}\label{DUALMONdef}
Let $\emptyset \ne N \subset B^*$.  We say that $N$ is {\em monotone} if, for all $d^*,e^* \in N$, $\qt(d^* - e^*) \ge 0$.
\end{definition}
\begin{theorem}\label{AFMAXthm}
Let $M$ be of type (NI).  Then $M^\sharp$ is a maximally monotone subset of $B^*$.
\end{theorem}
\begin{proof}Let $d^*,e^* \in M^\sharp$.   From \eqref{MG}, $G_M(d^*),G_M(e^*) \le 0$.   Setting $\lambda = \mu = \half$ in \eqref{G1}, $\qt(e^* - d^*) \ge 0$, and so $M^\sharp$ is a monotone subset of $B^*$.  It remains to prove that $M^\sharp$ is {\em maximally} monotone.   So we suppose that   
\begin{equation*}
b^* \in B^*\hbox{ and, } \all\ a^* \in M^\sharp,\ \qt\big(a^* - b^*\big) \ge 0
\end{equation*}
and we will establish that $b^* \in M^\sharp$.  From  \eqref{AFMAX1}, for all $m \in M$, $\qt\big(Lm - b^*\big) \ge 0$. Thus, from \eqref{G}, $G_M(b^*) \le 0$ and, from \eqref{MG}, $b^* \in M^\sharp$.   This completes the proof of \Thm{AFMAXthm}.
\end{proof}
\begin{corollary}\label{SURlem}
Suppose that,  for all $x^* \in E^*$, there exists $x \in E$ such that $(x,x^*) \in M$.   {\em ($M$ is the graph of a {\em surjective} multifunction from $E$ onto $E^*$.)}   Then $M$ is of type (NI).   If, further, $L(M)$ is a maximally monotone subset of $B^*$ then $L(M) = M^\sharp$.
\end{corollary}
\begin{proof}
Let $b^* = (x^*,x\dbs)$ be an arbitrary element of $B^*$.   By hypothesis,  there exists $x \in E$ such that $m = (x,x^*) \in M$. 
Then \eqref{G} implies that   
\begin{align*}
G_M(b^*) &\ge - \qt(Lm - b^*) = - \qt((x^*,\wh x) - (x^*,x\dbs)) = - \qt(0,\wh x - x\dbs) = 0,
\end{align*}   
and thus it follows from \eqref{NIdef2} that $M$ is of type (NI).  \big(This argument is taken from \cite[Corollary 7.9, p.\ 1034]{PARTONE}.\big)   From \eqref{AFMAX1} and \Thm{AFMAXthm}, $L(M) \subset M^\sharp$ and  $M^\sharp$ is (maximally) monotone.   Consequently, if $L(M)$ is {\em maximally} monotone then $L(M) = M^\sharp$.
\end{proof}
\begin{example}\label{TAILGex}
Define $U\colon\ \ell_1 \to c_0$ by $(Ux)_n = \sum_{k \ge n} x_k$.   $U^{-1}$ is, of course, the subset $\big\{(Ux,x)\big\}_{x \in \ell_1}$ of  $B = c_0 \times \ell_1$.  Then $U^{-1}$ is maximally monotone of type (NI), but $\big({U^{-1}}\big)^\sharp = G(T)$, where $T$ is the {\em tail operator} of \Rem{TAILrem}.   From\break \Ex{TAILNIex}, ${\big({U^{-1}}\big)}^\sharp$ is a maximally monotone subset of $B^* = \ell_1 \times \ell_\infty$, but not of type (NI).
\end{example}
\begin{proof}
  Clearly $L(U^{-1}) = G(T)$.  From Fact \ref{FOLKLORE}, $L(U^{-1})$ is a maximally monotone subset of $B^* = \ell_1 \times \ell_\infty$. It follows easily from \eqref{not7} that $U^{-1}$ is a monotone subset of $c_0 \times \ell_1$.   We now prove that $U^{-1}$ is a maximally monotone subset of $c_0 \times \ell_1$, so that we can apply \Cor{SURlem}.   To this end, let $b \in c_0 \times \ell_1$ and $\inf q(U^{-1} - b) \ge 0$.   From \eqref{not7}, $\inf \qt\big(L(U^{-1}) - Lb\big) \ge 0$.   Since $L(U^{-1})$ is maximally monotone, $Lb \in L(U^{-1})$, from which $b \in U^{-1}$.   This completes the proof of the maximal monotonicity of $U^{-1}$.   \Cor{SURlem} implies that $U^{-1}$ is of type (NI) and ${\big({U^{-1}}\big)}^\sharp = L(U^{-1}) = G(T)$.   From \Ex{TAILNIex}, ${\big({U^{-1}}\big)}^\sharp$ is not of type (NI).
\end{proof}
\begin{theorem}[Fundamental property of type (NI) sets]\label{AFRLthm}
Let $M$ be of type (NI) and $c \in B$.  Then there exists $b^* \in M^\sharp$ such that $\rt(Lc - b^*) = 0$.
\end{theorem}
\begin{proof}
From \eqref{NIdef2} and \Lem{Glem}(b), $F_M \ge 0\ \on\ B^*$ so, from \Cor{Signscor}(b), $(F_M \squaree \rt)(Lc) = 0$, that is to say $\minn_{b^* \in B^*}\big[F_M(b^*) + \rt(Lc - b^*)\big]= 0$.   Since $\rt \ge 0$ on $B^*$ also, there exists $b^* \in B^*$ such that $F_M(b^*) = 0$ and $\rt(Lc - b^*) = 0$.   From \Lem{Glem}(b) again, $G_M(b^*) \le 0$, and the result follows from \eqref{MG}.   
\end{proof}
The result of \Thm{AFRLthm} take on a more familiar aspect if we work in terms of {\em multifunctions} rather than {\em subsets of a product space}.   Write $\J \colon E^* \toto E\dbs$ for the duality map, so $z\dbs \in {\J}z^*$ exactly when $\half\|z^*\|^2 + \half\|z\dbs\|^2 - \bra{z^*}{z\dbs} = 0$.   From \eqref{rtdef}, we have $ z\dbs \in \J z^* \iff \rt(-z^*,z\dbs) = 0 \iff \rt(z^*,-z\dbs) = 0 \iff z^* \in  \J^{-1} z\dbs$.   
\begin{corollary}[(NI) surjectivity theorems]\label{NISURcor}
Let $S\colon E \toto E^*$ and $G(S)$ be maximally monotone of type (NI).  Define $S^\sharp\colon E^* \toto E\dbs$ so that $G(S^\sharp) = G(S)^\sharp$ in the sense of \eqref{MG}.   Then $S^\sharp$ is maximally monotone.   Now let $x \in E$ and $x^* \in E^*$.   Then there exist $y^* \in E^*$ and $y\dbs \in E\dbs$ such that
\begin{equation}\label{NISUR1}
\wh x \in {S^\sharp}y^* + {\J}(y^* - x^*)\quand x^* \in (S^\sharp)^{-1}(y\dbs) + \J^{-1}(y\dbs - \wh x).
\end{equation}
Furthermore,
\begin{equation}\label{NISUR2}
R\big({S^\sharp} + {\J}\big) \supset \wh E \quand R\big((S^\sharp)^{-1} + {\J}^{-1}\big) = E^*.
\end{equation}  
\end{corollary}
\begin{proof} The maximal monotonicity of $S^\sharp$ is immediate from \Thm{AFMAXthm} with $M := G(S)$.
\par
From \Thm{AFRLthm} with $c = (x,x^*)$, there exist $y^* \in E^*$ and $y\dbs \in S^\sharp y^*$ such that  $\rt(x^* - y^*, \wh x - y\dbs) = 0$.   But then $\wh x - y\dbs \in \J(y^* - x^*)$ and so the first observation in \eqref{NISUR1} follows since $\wh x = y\dbs + (\wh x - y\dbs)$.   We also have $y^* \in (S^\sharp)^{-1}(y\dbs)$ and $x^*  - y^*\in \J^{-1}(y\dbs - \wh x)$ and so the second observation in \eqref{NISUR1} follows since $x^* =  y^* + (x^*  - y^*)$.
\par
\eqref{NISUR2} follows from the two observations in \eqref{NISUR1} by taking (respectively)\break $x^* = 0$ and $x = 0$. 
\end{proof}
We now take a brief digression into the case where $E$ is reflexive.
\begin{theorem}\label{NIrefl}
Let $E$ be a reflexive space and $M$ be a maximally monotone subset of $B$.   Then $M$ is of type (NI) and  $M^\sharp = L(M)$. 
\end{theorem}
\begin{proof}Let $b^* \in B^*$. Then there exists $b \in B$ so that $Lb = b^*$, and we obtain from \Lem{Glem}(a) and \eqref{P} that $G_M(b^*) = G_M(Lb) = P_M(b) \ge 0$.   Thus \eqref{NIdef2} is true, and $M$ is of type (NI).   Suppose now that $b^* \in M^\sharp$.   Using the argument above and \eqref{MG}, there exists $b \in B$ so that $Lb = b^*$ and $P_M(b) = G_M(b^*) \le 0$, and so \eqref{P2} implies that $b \in M$.   Thus $M^\sharp \subset L(M)$, and it follows from \eqref{AFMAX1} that $M^\sharp = L(M)$.     
\end{proof}
\begin{corollary}[Rockafellar's surjectivity theorem for reflexive spaces]\label{ROCKSURcor}
Let $E$ be reflexive, $S\colon\ E \toto E^*$ be maximally monotone and $J\colon\ E \toto E^*$ be the duality map.   Then $R(S + J) = E^*$.
\end{corollary}
\begin{proof}
From \Thm{NIrefl}, $G(S)$ is of type (NI), and the second observation in \eqref{NISUR2} and the reflexivity of $E$ imply that, for all $x^* \in E^*$, there exists $x \in E$ such that $(S^\sharp)^{-1}(\wh x) + {\J}^{-1}(\wh x) = x^*$.   From \Thm{NIrefl},
$$z^* \in (S^\sharp)^{-1}(\wh x) \iff (z^*,\wh x) \in L(G(S)) \iff (x,z^*) \in G(S) \iff z^* \in Sx.$$
It is also easy to see that ${\J}^{-1}(\wh x) = Jx$, and so $Sx + Jx = x^*$. 
\end{proof}
\begin{remark}
There is a discussion of the proofs of Rockafellar's surjectivity theorem for reflexive spaces in \cite[Remark 29.4, p.\ 118]{HBM}.   See also \Rem{SURrem} for a discussion of the connection between \eqref{NISUR2} and a result of Gossez.  
\end{remark}
\section{Linear subspaces and their adjoints}\label{LINsec}
In this section, $\CML(B)$ and $\CML(B^*)$ stand for the set of closed monotone linear subspaces of $B$ and $B^*$, and $\MML(B)$ and $\MML(B^*)$ for the set of\break maximally monotone linear subspaces of $B$ and $B^*$.   Let $V \in \CML(B)$.\break  The main result is \Thm{qqthm}, which extends \cite[Theorem 6.4, p.\ 269]{POLAR} and \cite[Corollary 6.6, p.\ 269]{POLAR}. By virtue of \big((a)$\iff$(b)\big) of \Thms{EQUIVthm} and \ref{Dthm}, \Thm{qqthm}\big((c)$\iff$(d)\big) is equivalent to the result first proved in\break \cite[Theorem 4.1, p.\ 4960]{BBWYBB}.
\par
We first investigate the element $g_d$ of $\PCLSC(B)$ defined by the parameters of the problem,
and compute $\partial g_d$ in \Lem{subglem}. \Lems{ALHORlem}--\ref{hlem} are devoted to some convex analysis: \Lem{ALHORlem} uses the Br\o ndsted--Rockafellar theorem, and \Lem{hlem} uses Rockafellar's formula for the subdifferential of a sum.  \Lem{FNLlem}(a,d) are devoted to the automatic maximality of $V$ and $V^\A$ (see below) under certain circumstances.   Finally, \Lem{FNLlem} leads rapidly to\break \Thm{qqthm}.         
\par
The {\em adjoint subspace}, $V^\A$, of $V$ is the closed linear subspace of $B^*$ defined by
\begin{equation}\label{OLDADJdef}
(y^*,y\dbs) \in V^\A \iff \all\ (x,x^*) \in V,\ \bra{x}{y^*} = \bra{x^*}{y\dbs}.  
\end{equation}
Note that $(y^*,y\dbs) \in V^\A \iff (y\dbs,y^*) \in V^*$ in the sense of \cite{BBWYBB}. We define the bilinear map $\xbra{\cdot}{\cdot} \colon\ B \times B^* \to \RR$ by: 
$$\xbra{(x,x^*)}{(y^*,y\dbs)}:=  \bra{x}{y^*} - \bra{x^*}{y\dbs}.$$ 
Clearly,
\begin{equation}\label{XADJdef}
a^* \in V^\A \iff \Xbra{V}{a^*} = \{0\}.  
\end{equation}
We define the {\em reflection map} $\rho\colon\ B^* \to B^*$ by  $\rho(y^*,y\dbs) := (y^*,-y\dbs)$.
Then
\begin{equation}\label{rhoprops}
\left.
\begin{gathered}
\qt \circ \rho = -\qt\hbox{ and }\\
 (a,a^*) \in B \times B^* \lr \xbra{a}{a^*} = \bra{a}{\rho a^*}\hbox{ and } \xbra{a}{\rho a^*} = \bra{a}{a^*}.  
\end{gathered}
\right\}
\end{equation}
\par
It is easy to see that if $W$ is a linear subspace of $B$ (resp. $B^*$) then
\begin{equation}\label{Wmon}
W\hbox{ is monotone }\iff q(W) \subset [0,\infty[\enspace\big(\hbox{resp. } \qt(W) \subset [0,\infty[\,\big). 
\end{equation}
We will use the following standard notation:
\begin{definition}\label{SUBdef}
If $f\colon\ B \to \rbar$ and $b \in \dom\,f$ then the {\em subdifferential of $f$ at $b$}, $\partial f(b)$, is the subset of $B^*$ defined by
$$b^* \in \partial f(b) \iff \all\ c \in B,\ \bra{c - b}{b^*} \le f(c) - f(b).$$
\end{definition}   
\begin{lemma}\label{subglem}
Let $V \in \CML(B)$, $d \in B$, $g_d:= q + \I_{V - d}$ and $b \in \dom\,g_d = V - d$.   Then $g_d \in \PCLSC(B)$ and
\begin{equation}\label{subg1}
b^* \in \partial g_d(b) \iff \bra{V}{Lb - b^*} = \{0\}.
\end{equation}
If, further, $b^* \in \partial g_d(b)$  and  $V^\A \in \CML(B^*)$, then $q(b) + \qt(b^*) \le \bra{b}{b^*}$.
\end{lemma}
\begin{proof}
Let $c \in \dom\,g_d = V - d$, $\lambda, \mu \ge 0$ and $\lambda + \mu = 1$. Then $\lambda b + \mu c \in V - d$, so $g_d(\lambda b + \mu c) = q(\lambda b + \mu c)$.   From \eqref{Wmon}, $q(b - c) \ge 0$  and so, from \eqref{not2},
$$g_d(\lambda b + \mu c) \le q(\lambda b + \mu c) + \lambda\mu q(b - c) = \lambda q(b) + \mu q(c) = \lambda g_d(b) + \mu g_d(c).$$
Consequently, $g_d$ is convex.  From \eqref{qcont}, $q$ is continuous.   Since $V$ is closed, $g_d$ is lower semicontinuous.   Thus $g_d \in \PCLSC(B)$.
\par
We now prove the implication ``$\rl$'' in \eqref{subg1}.  So let $\bra{V}{Lb - b^*} = \{0\}$ and $v := c - b \in (V - d) - (V - d) = V - V = V$.  From \eqref{Wmon},
\begin{equation}\label{subg3}
q(v) \ge 0\hbox{\quad and \quad}\bra{v}{Lb - b^*} = \{0\}. 
\end{equation}
From \eqref{subg3}, the fact that $c = v + b$ and \eqref{not1},
\begin{equation}\label{subg4}
\bra{v}{Lb} = \bra{v}{b^*}\hbox{\quad and \quad}q(c) - q(b) = q\big(v + b\big) - q(b) = q(v) + \bra{v}{Lb}.
\end{equation}
From \eqref{subg3} and \eqref{subg4}, $g_d(c) - g_d(b) = q(v) + \bra{v}{Lb} \ge \bra{v}{Lb} = \bra{v}{b^*} =  \bra{c - b}{b^*}$.   Since this holds for all $c \in \dom\,g_d = V - d$, $b^* \in \partial g_d(b)$.   This completes the proof of the implication ``$\rl$'' in \eqref{subg1}.

We now prove the implication ``$\lr$'' in \eqref{subg1}.   So let $b^* \in \partial g_d(b)$.   Let $v$ be an arbitrary element of $V$ and $\lambda$ be an arbitrary element of $\RR$.   Let $c = b + \lambda v$.   Since $b \in V - d$, $c \in V - d$ also.  It follows from \Def{SUBdef}, and \eqref{not1} that   
\begin{align*}
\lambda\bra{v}{b^*} &= \bra{c - b}{b^*} \le g_d(c) - g_d(b) = q(c) - q(b)\\
&= q(b + \lambda v) - q(b) = \lambda\bra{v}{Lb} + \lambda^2q(v).
\end{align*}
Consequently, $\lambda^2q(v) + \lambda\bra{v}{Lb - b^*} \ge 0$.   Since this holds for all $\lambda \in \RR$, $\bra{v}{Lb - b^*} = 0$.   Since $v$ is an arbitrary element of $V$, $\bra{V}{Lb - b^*} = \{0\}$, as required.

If now $b^* \in \partial g_d(b)$  and  $V^\A \in \CML(B^*)$ then, from what we have already proved,      $\bra{V}{Lb - b^*} = \{0\}$.   From \eqref{rhoprops} and \eqref{XADJdef}, $\xbra{V}{\rho(Lb - b^*)} = \{0\}$ and so $\rho(Lb - b^*) \in V^\A$.   Since $V^\A \in \CML(B^*)$, \eqref{Wmon} gives $\qt\circ\rho(Lb - b^*) \ge 0$ and, from \eqref{rhoprops} again, $\qt(Lb - b^*) \le 0$, and \eqref{not9} implies that $q(b) + \qt(b^*) \le \bra{b}{b^*}$. 
\end{proof}
\begin{lemma}[Subtangents with low slope]\label{ALHORlem}
Let $f \in \PCLSC(B)$, $\inf_Bf > -\infty$, and $\eta > 0$. Then there exist $b \in \dom\,f$ and $e^* \in \partial f(b)$ such that $\|e^*\| \le \eta$.
\end{lemma}
\begin{proof}
From Br\o ndsted--Rockafellar, \cite[Lemma, pp.\ 608--609]{BRON} with $x^* = 0$, if $f(x) - \inf_Bf \le \eps$ and $\lambda > 0$, there exist $b \in B$ and $e^* \in \partial f(b)$ such that  $\|e^*\| \le \eps/\lambda$.   So let $x \in \dom\,f$, $\eps > f(x) - \inf_Bf \ge 0$ and $\lambda > \eps/\eta > 0$. 
\end{proof}

\begin{lemma}\label{NORMlem}
Let $\eta > 0$, $b \in B$ and $c^* \in \partial(2\eta\|\cdot\|)(b)$. Then
\smallbreak
\centerline{ $\|c^*\| \le 2\eta$ and $\bra{b}{c^*} = 2\eta\|b\|$.}
\end{lemma}
\begin{proof}
Let $a \in B$.  Then\quad $\bra{a}{c^*} = \bra{a + b}{c^*} - \bra{b}{c^*} \le 2\eta\|a + b\| - 2\eta\|b\| \le 2\eta\|a \|$.\quad Since this holds for all $a \in B$, $\|c^*\| \le 2\eta$.   However, $-\half\bra{b}{c^*} = \bra{\half b - b}{c^*} \le 2\eta\|\half b\| - 2\eta\|b\| = - \eta\|b\|$, so $\bra{b}{c^*} \ge 2\eta\|b\|$, and the equality follows since $\|c^*\| \le 2\eta$. 
\end{proof}
\begin{lemma}\label{hlem}
Let $h \in \PCLSC(B)$, $\inf h(B) > -\infty$ and $\eta > 0$.   Then there exist $b \in \dom\,h$ and $b^* \in \partial h(b)$ such that $\|b^*\| \le 3\eta$ and $\bra{b}{b^*} \le -\eta\|b\|$.  
\end{lemma}
\begin{proof}
From \Lem{ALHORlem} with $f:= h + 2\eta\|\cdot\|$, there exist $b \in \dom\,h$ and $e^* \in \partial (h + 2\eta\|\cdot\|)(b)$ such that $\|e^*\| \le \eta$.   Since $\|\cdot\|$ is continuous, Rockafellar's formula for the subdifferential of a sum \big(Rockafellar, \cite[Theorem 3(b), p.\ 85]{FENCHEL}, Z\u alinescu, \cite[Theorem 2.8.7(iii), p.\ 127]{ZBOOK}, or \cite[Theorem 18.1, pp.\ 74--75]{HBM}\big) provide $b^* \in \partial h(b)$ and $c^* \in \partial(2\eta\|\cdot\|)(b)$ such that $b^* + c^* = e^*$.   But then $b^* = e^* - c^*$ and so, from \Lem{NORMlem},
\begin{equation*}
\|b^*\| \le \|e^*\| + \|-c^*\| \le \eta + 2\eta \quand \bra{b}{b^*} = \bra{b}{e^*} - \bra{b}{c^*} \le \eta\|b\| - 2\eta\|b\|.
\end{equation*}
The result now follow easily.  
\end{proof}
\begin{lemma}\label{FNLlem}
Let $V \in \CML(B)$ and $V^\A \in \CML(B^*)$.   Then:\par\noindent
{\rm(a)}\enspace $V \in \MML(B)$.\par\noindent
{\rm(b)}\enspace For all $d^* \in B^*$, $\infn\qt\big(L(V) - d^*\big) \le 0$.\par\noindent
{\rm(c)}\enspace $w \in V$, $b^* \in B^*$ and $\xbra{w}{b^*} \ne 0 \lr \inf\qt(V^\A - b^*) < 0$.\par\noindent
{\rm(d)}\enspace $V^\A\in \MML(B^*)$.
\end{lemma}
\begin{proof}
(a)\enspace We suppose that $d \in B$ and $\inf q(V - d) \ge 0$, and we shall prove that $d\in V$.   From \Lem{subglem}, $g_d \in \PCLSC(B)$ and $\infn_Bg_d  \ge 0$.
Let $\eta > 0$. From \Lem{hlem} with $h := g_d$, there exist $b \in  V - d$ and $b^* \in \partial g_d(b)$ such that
\begin{equation}\label{FNL1}
\|b^*\| \le 3\eta\hbox{\quad and\quad}\bra{b}{b^*} \le -\eta\|b\|.
\end{equation} 
It follows from this that
\begin{equation}\label{FNL2}
- \qt(b^*) =  - \half\Bra{b^*}{\Lt b^*}  \le \half\|b^*\|^2 \le \textstyle\frac{9}{2}\eta^2 < 5\eta^2.
\end{equation}
Since $b^* \in \partial g_d(b)$, \Lem{subglem} and \eqref{FNL1} give $q(b) + \qt(b^*) \le -\eta\|b\|$. Since $b \in V - d$, $q(b) \ge 0$ and, from \eqref{FNL2}, $\qt(b^*) \ge - 5\eta^2$.   Thus $-5\eta^2 \le -\eta\|b\|$, from which $\|b\| \le 5\eta$.   Since $\eta$ can be taken arbitrarily small and $V - d$ is closed, $0 \in V - d$, and so $d \in V$, as required.
\par
(b)\enspace If $\infn\qt\big(L(V) - d^*\big) = -\infty$, the result is obvious.   So we can and will suppose that $\infn\qt\big(L(V) - d^*\big) > -\infty$  
Let $h:= g_0 - d^*$.   From \Lem{subglem}, $h \in \PCLSC(B)$. Let $\eta > 0$.   Then it follows from \eqref{not9} that  
$$\inf h(B) = \inf(q - {d^*})(V) =\infn\qt\big(L(V) - d^*\big)- \qt(d^*) > -\infty.$$
From \Lem{hlem}, there exist $b \in \dom\,h = V$ and $a^* \in \partial h(b)$ such that\break $\|a^*\| \le 3\eta$ and $\bra{b}{a^*} \le -\eta\|b\| \le 0$.   The definition of $h$ gives us $b^* \in \partial g_0(b)$ such that $a^* = b^* - d^*$.   Thus
\begin{equation}\label{FNL4}
\|b^* - d^*\| \le 3\eta \quand -\bra{b}{d^*} \le -\bra{b}{b^*},
\end{equation}
from which\quad $\|b^* + d^*\| \le \|b^* - d^*\| + 2\|d^*\| \le 3\eta + 2\|d^*\|$\quad and so
\begin{align*}
\qt(d^*) - \qt(b^*) &= \half\Bra{d^*}{\Lt d^*} - \half\Bra{b^*}{\Lt b^*} = \half\Bra{d^* - b^*}{\Lt(d^* + b^*)}\\
&\le \half\|b^* - d^*\|\|b^* + d^*\| \le \half 3\eta(3\eta + 2\|d^*\|),
\end{align*}
from which
%
\begin{equation}\label{FNL3}
\qt(d^*) \le \qt(b^*) + \textstyle\frac{9}{2}\eta^2 + 3\|d^*\|\eta.
\end{equation}
From \eqref{not9} with $b^*$ replaced with $d^*$, \eqref{FNL4} and \eqref{FNL3}, 
$$\qt(Lb - d^*) = q(b) - \bra{b}{d^*} + \qt(d^*) \le q(b) - \bra{b}{b^*} + \qt(b^*) + \ts\frac{9}{2}\eta^2 + 3\|d^*\|\eta.$$
Since $b^* \in \partial g_0(b)$, \Lem{subglem} with $d = 0$ gives $q(b) + \qt(b^*) \le \bra{b}{b^*}$, so 
$$\qt(Lb - d^*) \le 0 + \ts\frac{9}{2}\eta^2 + 3\|d^*\|\eta = \ts\frac{9}{2}\eta^2+ 3\|d^*\|\eta.$$
Since $b \in V$, the result now follows by making $\eta$ arbitrarily small.
\par
(c)\enspace Let $w \in V$, $b^* \in B^*$ and $\xbra{w}{b^*} \ne 0$. We shall prove that
\begin{equation}\label{Sec1}
\inf\qt(V^\A - b^*) < 0.
\end{equation}
Let 
\begin{equation}\label{Sec2}
{U}:= \big\{u \in V\colon\ \xbra{u}{b^*} = 0\big\} \subset V \setminus \{w\}.
\end{equation}
${U} \in \CML(B)$.  Since $U$ is a proper subspace of $V$, $U \not\in \MML(B)$.   Applying (a), with $V$ replaced by $U$, $U^\A \not\in \CML(B^*)$, and so \eqref{Wmon} with $W = U^\A$ provides
\begin{equation}\label{Sec3}
c^* \in {U}^\A\ \st\ \qt(c^*) < 0.
\end{equation}
Let $v$ be an arbitrary element of $V$.   Let $u := \xbra{v}{b^*}w - \xbra{w}{b^*}v$. Since $u$ is a linear combination of $w,v \in V$, we also have $u \in V$.   Furthermore, $\xbra{u}{b^*} = \xbra{v}{b^*}\xbra{w}{b^*} - \xbra{w}{b^*}\xbra{v}{b^*} = 0$ and so, from \eqref{Sec2}, $u \in U$.   Thus, from
\eqref{Sec3} and \eqref{XADJdef} (applied to $U$),
\begin{align*}
\Xbra{v}{\xbra{w}{c^*}b^* - \xbra{w}{b^*}c^*} &= \xbra{v}{b^*} \xbra{w}{c^*} - \xbra{w}{b^*}\xbra{v}{c^*}\\
&= \Xbra{\xbra{v}{b^*}w - \xbra{w}{b^*}v}{c^*} = \xbra{u}{c^*} = 0.
\end{align*}
Since this holds for all $v\in V$,
\begin{equation}\label{Sec5}
\xbra{w}{c^*}b^* - \xbra{w}{b^*}c^* \in V^\A.
\end{equation}
From \eqref{Sec3}, $\qt(-\xbra{w}{b^*}c^*) = \xbra{w}{b^*}^2\qt(c^*) < 0$, and so \eqref{Wmon} with $W = V^\A$ implies that $- \xbra{w}{b^*}c^* \not\in V^\A$.   Thus, from \eqref{Sec5}, $\xbra{w}{c^*} \ne 0$.   Now set $\lambda := \xbra{w}{c^*} \neq 0$ and $\mu := \xbra{w}{b^*} \neq 0$, so  
\eqref{Sec5} reads $\lambda b^* - \mu c^* \in V^\A$.   Dividing by $\lambda$, $b^* - (\mu/\lambda) c^* \in V^\A$.
From \eqref{Sec3}, $\inf\qt(V^\A - b^*) \le \qt\big(b^* -  (\mu/\lambda)c^* - b^*\big)  = {(\mu/\lambda)^2}\qt(c^*) < 0$.   This establishes \eqref{Sec1}, and completes the proof of (c).\par
(d) It is immediate from (c) that 
$$b^* \in B^*,\ \inf\qt(V^\A - b^*) \ge 0 \hbox{ and }w \in V \lr \xbra{w}{b^*} = 0.$$
So let $b^* \in B^*$ and $\inf\qt(V^\A - b^*) \ge 0$.   Then $\xbra{V}{b^*} = \{0\}$ and so, from \eqref{XADJdef}, $b^* \in V^\A$.   This completes the proof of (d).
\end{proof}
A remark about \Thm{qqthm} is in order.   If we knew that $V$ is {\em maximally} monotone then \Thm{qqthm}(a) is exactly that $V$ be of type (NI).   We shall see in  \Thm{qqthm}(c) that, when  $V \in \CML(B)$, the condition (a) of \Thm{qqthm} actually {\em implies} that $V$ {\em is} maximally monotone.

\begin{theorem}\label{qqthm}
Let $V \in \CML(B)$.   Then the following 5 conditions on $V$ are equivalent:
\par\noindent
{\rm(a)}\enspace For all $b^* \in B^*$, $\infn\qt\big(L(V) - b^*\big) \le 0$.\par\noindent
{\rm(b)}\enspace $\infn \qt(V^\A) \ge \infn q(V)$.\par\noindent
{\rm(c)}\enspace $V^\A \in \CML(B^*)$.\par\noindent
{\rm(d)}\enspace $V \in \MML(B)$ and $V$ is of type (NI).\par\noindent 
{\rm(e)}\enspace $V^\A \in \MML(B^*)$.
\end{theorem}
\begin{proof}
Suppose first that (a) is true and $a^* \in V^\A$.   From \eqref{XADJdef}, \eqref{rhoprops} and \eqref{not9},
\begin{align*}
\inf q(V) - \qt(a^*) &=\infn_{v \in V}[q(v) + \qt(\rho a^*)]
= \infn_{v \in V}[q(v) - \xbra{v}{a^*} + \qt(\rho a^*)]\\
&= \infn_{v \in V}[q(v) - \bra{v}{\rho a^*} + \qt(\rho a^*)]  = \infn_{v \in V}\qt(Lv - \rho a^*) \le 0.
\end{align*}
Thus $\qt(a^*) \ge \inf q(V)$.  Consequently, (b) is true.\par
Suppose now that (b) is true.   From \eqref{Wmon}, $\infn q(V) \ge 0$, thus $\infn \qt(V^\A) \ge 0$.  From \eqref{Wmon}, (c) is true.
\par
It is immediate from \Lem{FNLlem}(a,b) that (c)$\lr$(d) and, from\break \Defs{Gdef} and \ref{NIdef}, that (d)$\lr$(a).   Thus (a)--(d) are equivalent.
\par
Finally, it is clear from \Lem{FNLlem}(d) that (c)$\iff$(e).  
\end{proof}
\begin{remark}\label{BBrem}
By virtue of \Thm{NIrefl}, \Thm{qqthm} generalizes the\break result proved in Brezis-Browder \cite[Theorem 2]{BB}: {\em Let $E$ be a reflexive Banach space and $V \in \CML(B)$.   Then $V \in \MML(B) \iff$  $V^\A \in \CML(B^*)$.}
\end{remark}
\section{Two sets of equivalences}\label{EQUIVsec}
We recall that $M$ is a maximally monotone subset of $B$.
\begin{theorem}\label{EQUIVthm}
The following five conditions are equivalent:\par\noindent
{\rm(a)}\enspace $M$ is quasidense.\par\noindent
{\rm(b)}\enspace $G_M \ge 0$ on $B^*$, \hbox{\rm\ i.e.,  $M$ is of type (NI).}\par\noindent
{\rm(c)}\enspace $F_M \ge 0$ on $B^*$.\par\noindent
{\rm(d)}\enspace $P_M \ssquare r \le 0$ on $B$.\par\noindent
{\rm(e)}\enspace $(F_M \ssquare \rt) \circ L \ge 0$ on $B$.
\end{theorem}
\begin{proof}
Suppose first that (a) is satisfied.  Let $b^* = (y^*,y\dbs) \in B^*$ and $\eps > 0$.  The definition of $\|y\dbs\|$ gives an element $z^*$ of $E^*$ such that $\|z^*\| \le \|y\dbs\|$ and $\bra{z^*}{y\dbs} \ge \|y\dbs\|^2 - \eps$.   Let $b:= (x,x^*) = (0,y^* + z^*) \in B$.   Then $Lb - b^* = (y^* + z^*,0) - (y^*,y\dbs) = (z^*,-y\dbs)$.
From \eqref{rtdef},
$$\rt(Lb - b^*) = \half\|z^*\|^2 + \half\|y\dbs\|^2 - \bra{z^*}{y\dbs} \le \|y\dbs\|^2 - \bra{z^*}{y\dbs} \le \eps.$$
The quasidensity of $M$ provides $m \in M$ such that $r(m - b) < \eps$.  Since $\big\|\Lt\big\| \le 1$, from \eqref{not3} with $\lambda = 1$, $\mu = -1$,  $d^* = Lm - Lb$, $e^* = Lb - b^*$,  \eqref{rtdef} and \eqref{not7},
\begin{align*}
-G_M(b^*) &\le \qt(Lm - b^*) \le \qt(Lm - Lb) + \|Lm - Lb\|\|Lb - b^*\| + \qt(Lb - b^*)\\
 &\le \qt(Lm - Lb) + \half\|Lm - Lb\|^2 + \half\|Lb - b^*\|^2 + \qt(Lb - b^*)\\
 &= \rt(Lm - Lb) + \rt(Lb - b^*) = r(m - b) + \rt(Lb - b^*) < 2\eps.
\end{align*}
Thus $G_M(b^*) \ge 0$, and so (b) follows from \eqref{NIdef2}.

It is clear from  \Lem{Glem}(b) that (b)$\lr$(c), from \Cor{Signscor}(a) that (c)$\lr$(d) and, from \Thm{SUFFthm}(b), that (d)$\lr$(a).   Thus (a)--(d) are equivalent, and \Thm{Signsthm} gives (d)$\iff$(e).
\end{proof}
\begin{remark}
We will give a sixth and seventh condition equivalent to those in  \Thm{EQUIVthm} in \Thm{Dthm}.   
\end{remark}
We now give nine characterizations of the Gossez extension of a \cmqd\ subset of $B$.   The equivalence of (a) and (g) below will be used in \Lem{PAIRSlem}. 
\begin{theorem}\label{EQthm}
Let $M$ be a \cmqd\ subset of $B$ and $b^* \in B^*$.   Then the following ten conditions are equivalent:\par\noindent
{\rm(a)}\enspace $G_M(b^*) \le 0$,\hbox{\rm\ i.e.,  $b^* \in M^\sharp$.}\par\noindent
{\rm(b)}\enspace $G_M(b^*) = 0$.\par\noindent
{\rm(c)}\enspace $(G_M \ssquare\qt)(b^*) = 0$.\par\noindent
{\rm(d)}\enspace $(G_M \ssquare \qt)(b^*) \ge 0$.\par\noindent
{\rm(e)}\enspace $F_M(b^*) \le 0$.\par\noindent
{\rm(f)}\enspace $F_M(b^*) = 0$.\par\noindent
{\rm(g)}\enspace ${\Phi_M}\dbs\big(\Lt b^*\big) = \qt(b^*)$.\par\noindent
{\rm(h)}\enspace ${\Phi_M}\dbs\big(\Lt b^*\big) \le \qt(b^*)$.\par\noindent
{\rm(i)}\enspace $(F_M \ssquare \qt)(b^*) = 0$.\par\noindent
{\rm(j)}\enspace $(F_M \ssquare \qt)(b^*) \ge 0$.
\end{theorem}
\begin{proof}
It follows from \Thm{EQUIVthm}(a)$\lr$(b) and \eqref{FG2} that (a)$\lr$(b)$\lr$(c).    It is obvious that (c)$\lr$(d).  Now suppose that (d) is true.   From \eqref{DEL} and \Lem{Glem}(a),\quad $-F_M(b^*) = \infn_{b \in B}[P_M(b) + \qt(b^* - Lb)]$\quad and, for all $b \in B$,\quad $P_M(b) + \qt(b^* - Lb) = G_M(Lb) + \qt(b^* - Lb) \ge (G_M \ssquare \qt)(b^*) \ge 0$.\break Thus $-F_M(b^*) \ge 0$, and we have proved that (d)$\lr$(e).    It follows from\break \Thm{EQUIVthm}(a)$\lr$(c) that (e)$\lr$(f).   If (f) is true then, from \eqref{FG3}, (g) is true.   Obviously, (g)$\lr$(h).  It is immediate from \Lem{Flem}(c) that (g)$\lr$(i) and (h)$\lr$(j). Obviously (i)$\lr$(j).   If (j) is true then it follows from \Lem{FMLPMq} that, for all $m \in M$,\quad $\qt(b^* - Lm) = F_M(Lm) + \qt(b^* - Lm) \ge (F_M \ssquare \qt)(b^*) \ge 0$,\quad and so  \eqref{G} gives\quad    
$-G_M(b^*) = \inf\qt\big(L(M) - b^*\big) = \inf\qt\big(b^* - L(M)\big) \ge 0$.   Thus (j)$\lr$(a). This completes the proof of \Thm{EQthm}.
\end{proof}
\begin{remark}
A number of the results that we have proved lead to the\break conjecture that {\em if $M$ a \cmqd\ subset of $B$ then $G_M = F_M$ on $B^*$}. We do not know if this is true. 
\end{remark}
\section{Limiting results for certain biconjugates}\label{NETsec}
The main result in this section is \Thm{NETthm}, but we start by computing the biconjugates of three simple convex functions on $B$.
\begin{lemma}\label{CONJlem}
Let $w^* \in E^*$, $b^* = (z^*,z\dbs) \in B^*$, and define the continuous, convex functions $f_1,f_2$ and $h_{w^*}$ on $B$ by\quad $f_1(x,x^*) := \half\|x\|^2$,\quad $f_2(x,x^*) := \half\|x^* - z^*\|^2$\quad and\quad $h_{w^*}(x,x^*) := \bra{x}{w^*}$.\quad     Then we have\quad ${f_1}\dbs(\Lt b^*) = \half\|z\dbs\|^2$,\quad ${f_2}\dbs(\Lt b^*) = 0$\quad and\quad ${h_{w^*}}\dbs(\Lt b^*) = \bra{w^*}{z\dbs}$.
\end{lemma}
\begin{proof}
For all $(y^*,y\dbs) \in B^*$,
\begin{align*}
{f_1}^*(y^*,y\dbs)
&= \supn_{x \in E}\big[\bra{x}{y^*} - \half\|x\|^2]\big] + \supn_{x^* \in E^*}\bra{x^*}{y\dbs}\\
&=  \half\|y^*\|^2 + \I_{\{0\}}(y\dbs).
\end{align*}
Thus,\quad ${f_1}\dbs(\Lt b^*) =\supn_{y^* \in E^*}\big[\Bra{y^*}{z\dbs} - \half\|y^*\|^2\big] +  \supn_{\{0\}}{\wh{z^*}} = \half\|z\dbs\|^2$,\quad as required.   For all $(y^*,y\dbs) \in B^*$,
\begin{align*}
{f_2}^*(y^*,y\dbs)
&= \supn_{x \in E}\bra{x}{y^*} +  \supn_{x^* \in E^*}\big[\bra{x^*}{y\dbs} - \half\|x^* - z^*\|^2\big]\\
&= \supn_{x \in E}\bra{x}{y^*} +  \supn_{x^* \in E^*}\big[\bra{x^* + z^*}{y\dbs} - \half\|x^*\|^2\big]\\
&=  \I_{\{0\}}(y^*) + \half\|y\dbs\|^2 + \bra{z^*}{y\dbs}.
\end{align*}
Thus
\begin{align*}
{f_2}\dbs(\Lt b^*) &= \supn_{\{0\}}{z\dbs} + \supn_{y\dbs \in E\dbs}\big[\Bra{y\dbs}{\wh{z^*}} - \half\|y\dbs\|^2 - \bra{z^*}{y\dbs}\big]\\
&= \supn_{\{0\}}{z\dbs} + \supn_{y\dbs \in E\dbs}[ - \half\|y\dbs\|^2] = \supn_{\{0\}}{z\dbs} + \half\|0\|^2 = 0,
\end{align*}
as required.   Finally,
\begin{align*}
{h_{w^*}}^*(y^*,y\dbs) &= \supn_{x \in E}\bra{x}{y^* - w^*} + \supn_{x^* \in E^*}\bra{x^*}{y\dbs}\\
&= \I_{\{w^*\}}(y^*) + \I_{\{0\}}(y\dbs),
\end{align*}
so\quad ${h_{w^*}}\dbs(\Lt b^*) = \supn_{\{w^*\}}{z\dbs} + \supn_{\{0\}}\wh{z^*}  = \bra{w^*}{z\dbs}$,\quad
which completes the proof of \Lem{CONJlem}.
\end{proof}
\begin{lemma}\label{NEWFlem}
Let $k \ge 1$, $f_0 \in \PCLSC(B)$, $f_1,\dots,f_k$ be real, convex,\break continuous functions on $B$ and $n \ge 1$.   Let $b^* \in B^*$ and, for all $i = 0, \dots, k$, ${f_i}\dbs(\Lt b^*) \in \RR$.   Then there exists $b \in B$ such that, for all $i = 0, \dots, k$, $f_i(b) \le {f_i}\dbs(\Lt b^*) + \nth$.
\end{lemma}
\begin{proof}
For all $i = 0, \dots, k$, let $g_i := f_i - {f_i}\dbs(\Lt b^*)$, so ${g_i}\dbs(\Lt b^*) = 0$.   Let $g := \bigvee_{i = 0}^k{g_i} \in \PCLSC(B)$. From \cite[Corollary 45.5, p.\ 174]{HBM}, or alternatively \cite[Corollary~7, p.\ 3558]{FS}, $g\dbs(\Lt b^*) = \bigvee_{i = 0}^k{g_i}\dbs(\Lt b^*) = 0$.   However, we have $g\dbs(\Lt b^*) \ge \bra{0}{\Lt b^*} -  g^*(0) = \inf_Bg$, consequently $\inf_Bg \le 0$, and so there exists $b \in B$ such that $g(b) < \nth$.   This gives the desired result.  
\end{proof}
\begin{definition}\label{NWdef}
We write $\TNW$ for the  norm $\times\ w(E\dbs,E^*)$ topology on $B^*$.
\end{definition} 
\begin{theorem}\label{NETthm}
Let $f_0 \in \PCLSC(B)$, $b^* = (z^*,z\dbs) \in B^*$ and ${f_0}\dbs(\Lt b^*) \in \RR$.  Then there exists a net $b_\alpha = (x_\alpha,{x_\alpha}^*)$ of elements of $B$ such that:
\begin{gather}
\limsupn_\alpha f_0(b_\alpha) \le {f_0}\dbs(\Lt b^*).\label{NET1}\\
\limsupn_\alpha \|x_\alpha\|^2 \le \|z\dbs\|^2,\hbox{\quad from which\quad}\limsupn_\alpha \|x_\alpha\| \le \|z\dbs\|.\label{NET2}\\
\hbox{\rm(a) }\|x_\alpha^* - z^*\| \to 0\hbox{\quad and\quad \rm(b) }{\wh{x_\alpha}} \to {z\dbs}\hbox{ in } w(E\dbs,E^*).\label{NET3}\\
Lb_\alpha \to b^*\hbox{ in }{\TNW}.\label{NET4}\\
\hbox{\rm(a) }\limsupn_\alpha \|b_\alpha\| \le \|b^*\|\hbox{\quad and\quad \rm(b) }q(b_\alpha) \to \qt(b^*).\label{NET5}
\end{gather}
\end{theorem}
\begin{proof}
We say that the subset $W$ of $E^*$ is {\em symmetric} if
$$w^* \in W \qlr -w^* \in W.$$
Write $\W$ for the set of all finite symmetric subsets $W$ of $E^*$ such that $W \supset \{z^*,-z^*\}$.   $\W$ is directed (up) by inclusion.   Let $\NN := \{1,2,3.\dots\}$.   $\NN$ is directed (up) by the usual ordering.  Let $\D$ be the product directed set $\W \times \NN$,  and $\alpha = (W,n) \in \D$.  From \Lem{NEWFlem}, with $k$ the number of elements of $W$ increased by 2 and \Lem{CONJlem}, there exists $b_\alpha = (x_\alpha,{x_\alpha}^*) \in B$ such that
\begin{equation}\label{NET6}
\left.
\begin{gathered}
f_0(b_\alpha) \le {f_0}\dbs(\Lt b^*) + \nth,\\
\half\|x_\alpha\|^2 = f_1(b_\alpha) \le {f_1}\dbs(\Lt b^*) + \nth = \half\|z\dbs\|^2 + \nth,\\
\half\|x_\alpha^* - z^*\|^2 = f_2(b_\alpha)  \le {f_2}\dbs(\Lt b^*) + \nth = \nth\\
\hbox{and, for all }w^* \in W,\\
\bra{w^*}{\wh{x_\alpha}} = \bra{x_\alpha}{w^*} = h_{w^*}(b_\alpha) \le {h_{w^*}}\dbs(\Lt b^*)  + \nth = \bra{w^*}{z\dbs}  + \nth,
\end{gathered}
\right\} 
\end{equation}
which give \eqref{NET1}, \eqref{NET2} and \eqref{NET3}(a).
\par
Now let $w^* \in E^*$.   If $W \in \W$, $W \supset \{w^*,-w^*\}$ and $\alpha = (W,n) \in \D$ then, from \eqref{NET6},\quad  $\bra{w^*}{\wh{x_\alpha}} \le \bra{w^*}{z\dbs} + \nth$\quad and\quad $\bra{-w^*}{\wh{x_\alpha}} \le \bra{-w^*}{z\dbs} + \nth$,\quad from which\quad  $|\bra{w^*}{\wh{x_\alpha}} - \bra{w^*}{z\dbs}| \le \nth$.\quad   Thus\quad $\bra{w^*}{\wh{x_\alpha}} \to \bra{w^*}{z\dbs}$.\quad Since this holds for all $w^* \in E^*$, \eqref{NET3}(b) follows, completing the proof of \eqref{NET3}. Obviously, \eqref{NET3}$\iff$\eqref{NET4}. 
\par
We now establish \eqref{NET5}.   To this end,
\begin{align*}
|q(b_\alpha) - \qt(b^*)| &= |\bra{x_\alpha}{x_\alpha^*} - \bra{z^*}{z\dbs}| \le |\bra{x_\alpha}{x_\alpha^* - z^*}| + |\bra{x_\alpha}{z^*} - \bra{z^*}{z\dbs}|\\
&\le \|x_\alpha\|\|x_\alpha^* - z^*\| + |\bra{z^*}{\wh{x_\alpha} - z\dbs}|
\end{align*}
and \eqref{NET5}(b) follows from \eqref{NET2} and \eqref{NET3}.\quad We also obtain from \eqref{NET2} and \eqref{NET3} that $\limsupn_\alpha \|x_\alpha\|^2 \le \|z\dbs\|^2$ and $\limn_\alpha  \|x_\alpha^*\|^2 = \|z^*\|^2$, and \eqref{NET5}(a) follows by addition.   This completes the proof of \eqref{NET5}.    
\end{proof}
\section{Type \cmqd, type (D) and type (WD)}\label{Dsec}
We recall that $M$ is a maximally monotone subset of $B$.
\begin{lemma}\label{PAIRSlem}
Let $M$ be a \cmqd\ subset of $B$ and $b^* \in M^\sharp$.   Then there exists a net $m_\alpha= (s_\alpha,s_\alpha^*)$ of elements of $M$ such that
\begin{equation}\label{PAIRS1}
Lm_\alpha \to b^*\hbox{ in }{\TNW}\hbox{ i.e., }\|s_\alpha^* - z^*\| \to 0\hbox{ and }\wh{s_\alpha} \to z\dbs\hbox{ in }w(E\dbs,E^*),
\end{equation}
\begin{equation}\label{PAIRS2}
\hbox{and }\limsupn_\alpha \|m_\alpha\| \le \|b^*\|.
\end{equation}
\end{lemma}
\begin{proof}
\Lem{Flem}(a) and \Thm{EQthm}(a)$\lr$(g) imply that $\Phi_M \in \PCLSC(B)$ and ${\Phi_M}\dbs\big(\Lt b^*\big) = \qt(b^*) \in \RR$.   We now apply \Thm{NETthm} with $f_0 = \Phi_M$ and obtain a net $b_\alpha = (x_\alpha,{x_\alpha}^*)$ of elements of $B$ such that \eqref{NET1}--\eqref{NET5} are satisfied.   From \eqref{QD1}, there exists $m_\alpha = (s_\alpha,s_\alpha^*) \in M$ such that $r(m_\alpha - b_\alpha) \le \nth$.   \Lem{Plem}(b) and \Def{PHIdef} give
$$\half\|m_\alpha - b_\alpha\|^2 \le r(m_\alpha - b_\alpha) + P_M(b_\alpha) \le P_M(b_\alpha) + \nth = f_0(b_\alpha) - q(b_\alpha) + \nth.$$
Thus, from  \eqref{NET1} and \eqref{NET5}(b),
\begin{equation*}
\limsupn_\alpha \half\|m_\alpha - b_\alpha\|^2 \le \limsupn_\alpha f_0(b_\alpha) -  \limn_\alpha q(b_\alpha) \le {f_0}\dbs(\Lt b^*) - \qt(b^*) = 0.
\end{equation*}
It follows that  
\begin{equation}\label{PAIRS5}
\hbox{(a)}\|m_\alpha - b_\alpha\| \to 0,\hbox{ hence (b)} \|s_\alpha - x_\alpha\| \to 0\,\hbox{ and (c)} \|s_\alpha^* - x_\alpha^*\| \to 0.
\end{equation}
From \eqref{PAIRS5}(b),  $\wh{s_\alpha} -  \wh{x_\alpha} \to 0$ in $w(E\dbs,E^*)$.   Combining this with \eqref{NET3}(b),\break $\wh{s_\alpha} \to z\dbs$ in $w(E\dbs,E^*)$.   Combining \eqref{PAIRS5}(c) and \eqref{NET3}(a), $\|s_\alpha^* - z^*\| \to 0$.   This completes the proof of \eqref{PAIRS1}.   
\par
Finally, \eqref{PAIRS2} follows by combining \eqref{NET5}(a) and \eqref{PAIRS5}(a).
\end{proof}
%
\begin{definition}\label{Ddef}
In Phelps, \cite[Definition 3.1, pp.\ 216-217]{PRAGUE}, $M$ is defined to be of {\em type (D)} if, for every element, $b^* \in M^\sharp$, there exists a bounded net, $m_\alpha$, of elements of $M$ such that $Lm_\alpha \to b^*$ in $\TNW$.
\par
In \cite[Definition 14, p.\ 187]{RANGE}, $M$ is defined to be of {\em type (WD)} if, for every element $(z^*,z\dbs) \in M^\sharp$, there exists a bounded net, $(s_\alpha,s^*_\alpha)$ of elements of $M$ such that $\|s^*_\alpha - z^*\| \to 0$.

\end{definition} 
\begin{theorem}\label{Dthm}
The following three conditions are equivalent:
\par\noindent
{\rm(a)}\enspace $M$ is a \cmqd\ subset of $B$.
\par\noindent
{\rm(b)}\enspace $M$ is of type (D).
\par\noindent
{\rm(c)}\enspace $M$ is of type (WD).
\end{theorem}
\begin{proof}
If (a) is true and $b^* \in M^\sharp$, let $M_0 = \{m \in M\colon\ \|m\|\le \|b^*\| + 1\}$.   Then \Lem{PAIRSlem} gives a net in $M_0$ with the required property.   If (b) is true then it is obvious that (c) is true.
\par
Suppose, finally, that (a) is false.   From \Thm{EQUIVthm}\big((b)$\lr$(a)\big), there\break exists $b^* = (z^*,z\dbs) \in B^*$ such that $G_M(b^*) < 0$, and it follows from \eqref{MG} that $b^* \in M^\sharp$.   If (c) were true, there would exist $K \ge 0$ and a net, $(s_\alpha,s^*_\alpha) \in M$, such that $\|(s_\alpha,s^*_\alpha)\| \le K$ and $\|s^*_\alpha - z^*\| \to 0$.   From \eqref{G}, we wold have 
\begin{align*}
-G_M(b^*) &= \inf\qt(L(M) - b^*) \le \infn_\alpha\qt\big(L(s_\alpha,s^*_\alpha\big) - b^*)\\
&= \infn_\alpha\bra{s_\alpha^* - z^*}{\wh{s_\alpha} - z\dbs} \le \infn_\alpha\|s_\alpha^* - z^*\|\|\wh{s_\alpha} - z\dbs\|\\
&\le \infn_\alpha\|s_\alpha^* - z^*\|(K + \|z\dbs\|) = 0.
\end{align*}
This contradiction of the assumption above that $G_M(b^*) < 0$ completes the proof that (c)$\lr$(a).      
\end{proof}
\begin{remark}\label{SURrem}
In \cite[Theorems 3.5 and 3.6, pp.\ 218--221]{PRAGUE}, Phelps gives a proof of a result originally due to Gossez.   This proof is exceedingly difficult.     In the language of \Cor{NISURcor}, Gossez's result says:  {\em if $S$ is of type (D) then
$R\big((S^\sharp)^{-1} + {\J}^{-1}\big) = E^*$.}  (We have taken $\lambda = 1$ for simplicity).   Comparing this with the second assertion in \eqref{NISUR2}, and taking into account \Thm{Dthm}, it is clear that Gossez's result is subsumed by \Cor{NISURcor}.   
\end{remark}
\begin{remark}
By virtue of  \Thms{EQUIVthm} and \ref{Dthm}, the subset $M$ of $c_0 \times \ell_1$ introduced in \Ex{TAILGex} is maximally monotone of type (D) such that $M^\sharp$, though maximally monotone, is not of type (D).   The first example of this kind was given by Bueno--Svaiter in \cite[Proposition 3.2, p.\ 299]{OBS}.
\end{remark}

\begin{remark}
It is proved in Marques Alves--Svaiter, \cite[Theorem 4.4, pp.\ 1084--1085]{ASD} that {\em $M$ is of type (NI) if, and only if, $M$ is of type (D)}.   That proof is far from trivial.
\end{remark}

\end{document}